%% file: emb_codim_space_arcs_2021-1217.tex
\documentclass[a4paper,11pt,reqno]{amsart}

\usepackage[utf8]{inputenc}

\usepackage{etex}

\usepackage{xcolor}

\definecolor{verydarkblue}{rgb}{0,0,0.5}

\usepackage[
    breaklinks,
    colorlinks,
    citecolor=verydarkblue,
    linkcolor=verydarkblue,
    urlcolor=verydarkblue,
    pagebackref=true,
    hyperindex
]{hyperref}

\backrefenglish

\usepackage{fancyhdr}

\usepackage[
    hscale=0.7,
    vscale=0.75,
    headheight=13pt,
    centering,
]{geometry}

\usepackage{amsmath}
\usepackage{amsthm}
\usepackage{amssymb}
\usepackage{mathtools}
\usepackage{bm}		
\usepackage{mathdots}
\usepackage{framed}
\usepackage[capitalize]{cleveref}
\usepackage{array}
\usepackage[alphabetic,msc-links]{amsrefs}
\usepackage[all,cmtip]{xy}

\theoremstyle{plain}

\newtheorem{introtheorem}{Theorem}

\crefname{introtheorem}{Theorem}{Theorems}

\newtheorem{theorem}{Theorem}[section]
\newtheorem*{theorem*}{Theorem}
\newtheorem{proposition}[theorem]{Proposition}
\newtheorem{lemma}[theorem]{Lemma}
\newtheorem{corollary}[theorem]{Corollary}

\theoremstyle{definition}

\newtheorem{definition}[theorem]{Definition}

\theoremstyle{remark}

\newtheorem{remark}[theorem]{Remark}
\newtheorem{example}[theorem]{Example}
\newtheorem{question}[theorem]{Question}

\numberwithin{figure}{section}

\numberwithin{equation}{section}

\IfFileExists{./article-style.tex}{\input{article-style.tex}}{}

\usepackage{marginnote}

\def\N{{\mathbb N}}
\def\Z{{\mathbb Z}}
\def\Q{{\mathbb Q}}
\def\R{{\mathbb R}}
\def\C{{\mathbb C}}

\def\A{{\mathbb A}}
\def\P{{\mathbb P}}

\def\cI{\mathcal{I}}

\def\cP{\mathcal{P}}

\def\B{\mathcal{B}}

\def\O{\mathcal{O}}

\def\fa{\mathfrak{a}}
\def\fb{\mathfrak{b}}

\def\fm{\mathfrak{m}}
\def\fn{\mathfrak{n}}
\def\fp{\mathfrak{p}}
\def\fq{\mathfrak{q}}

\def\a{\alpha}
\def\b{\beta}
\def\g{\gamma}
\def\d{\delta}
\def\f{\varphi}
\def\ff{\psi}
\def\e{\eta}
\def\ep{\epsilon}

\def\m{\mu}

\def\p{\pi}

\def\s{\sigma}
\def\t{\tau}
\def\x{\xi}

\def\D{\Delta}

\def\Om{\Omega}
\def\vp{\varphi}

\def\.{\cdot}
\let\circum\^
\def\^{\widehat}
\def\~{\widetilde}
\def\o{\circ}
\def\ov{\overline}

\def\surj{\twoheadrightarrow}
\def\inj{\hookrightarrow}

\def\({\left(}
\def\){\right)}

\def\liminv{\underleftarrow{\lim}}

\def\*{{}^*}

\def\cotimes{\hat\otimes}

\renewcommand{\and}{ \ \ \text{ and } \ \ }

\def\sm{\mathrm{sm}}

\def\Jac{\mathrm{Jac}}

\def\MJ{\mathrm{MJ}}

\DeclareMathOperator{\codim} {codim}

\DeclareMathOperator{\rank} {rank}
\DeclareMathOperator{\im} {Im}

\DeclareMathOperator{\Gr} {Gr}

\DeclareMathOperator{\Spec} {Spec}
\DeclareMathOperator{\Spf} {Spf}

\DeclareMathOperator{\Sing} {Sing}

\DeclareMathOperator{\depth} {depth}
\DeclareMathOperator{\adj} {adj}

\DeclareMathOperator{\ord} {ord}

\DeclareMathOperator{\Sym} {Sym}

\DeclareMathOperator{\dirlim} {\varinjlim}
\DeclareMathOperator{\invlim} {\varprojlim}

\DeclareMathOperator{\ini} {in}

\DeclareMathOperator{\Fitt} {Fitt}

\DeclareMathOperator{\Hom} {Hom}

\DeclareMathOperator{\gr} {gr}

\DeclareMathOperator{\Tor} {Tor}

\DeclareMathOperator{\height} {ht}
\DeclareMathOperator{\mld} {mld}

\DeclareMathOperator{\Quot} {Quot}
\DeclareMathOperator{\trdeg}{trdeg}
\DeclareMathOperator{\grlex}{grlex}

\def\embdim{\mathrm{edim}}
\def\embcodim{\mathrm{ecodim}}
\def\fembcodim{\mathrm{fcodim}}

\def\isom{\simeq}

\begin{document}


\title{Embedding codimension of the space of arcs}

\author{Christopher Chiu}

\address[C.\ Chiu]{%
    Fakult\"at f\"ur Mathematik\\
    Universit\"at Wien\\
    Oskar-Morgenstern-Platz 1\\
    A-1090 Wien (\"Osterreich);\\[0.25em]
    and\\[0.25em]
    Department of Mathematics and Computer Science\\
    Eindhoven University of Technology\\
    De Groene Loper 5\\
    5612 AZ Eindhoven (Netherlands)%
}

\email{c.h.chiu@tue.nl}

\author{Tommaso de Fernex}

\address[T.\ de Fernex]{%
    Department of Mathematics\\
    University of Utah\\
    155 South 1400 East\\
    Salt Lake City, UT 48112 (USA)%
}

\email{defernex@math.utah.edu}

\author{Roi Docampo}

\address[R.\ Docampo]{%
    Department of Mathematics\\
    University of Oklahoma\\
    601 Elm Avenue, Room 423\\
    Norman, OK 73019 (USA)%
}

\email{roi@ou.edu}


\thanks{%
The research of the first author was partially supported by the Austrian
Science Fund (FWF) project P31338 and by the NWO Vici grant 639.033.514.
The research of the second author was partially supported by NSF Grants 
DMS-1700769 and DMS-2001254, and by NSF Grant DMS-1440140 while in residence at
MSRI in Berkeley during the Spring 2019 semester.
The research of the third author was partially supported by a grant from the
Simons Foundation (638459,~RD)%
}

\begin{abstract}
We introduce a notion of embedding codimension of an arbitrary local ring,
establish some general properties, and study in detail the case of arc spaces
of schemes of finite type over a field. Viewing the embedding codimension as a
measure of singularities, our main result can be interpreted as saying that the
singularities of the arc space are maximal at the arcs that are fully embedded
in the singular locus of the underlying scheme, and progressively improve as we
move away from said locus. As an application, we complement a theorem of
Drinfeld, Grinberg, and Kazhdan on formal neighborhoods in arc spaces by
providing a converse to their theorem, an optimal bound for the embedding
codimension of the formal model appearing in the statement, a precise formula
for the embedding dimension of the model constructed in Drinfeld's proof,
and a geometric meaningful way of realizing the decomposition stated in the
theorem. 
\end{abstract}


\maketitle


\section{Introduction}

\label{s:intro}

In this paper we define the \emph{embedding codimension} of an arbitrary local
ring and use it to quantify singularities of arc spaces. 

The embedding codimension is a familiar notion in the Noetherian setting,
where it is defined, for local rings, 
as the difference between the embedding dimension and the Krull dimension.
It was studied for instance in \cite{Lec64} under the name of \emph{regularity defect}.
Note, however, that if the ring is not Noetherian then both of these quantities can
be infinite, and even when they are finite it can happen that the embedding
dimension is smaller than the dimension. 
Rank-two valuation rings give simple examples where this phenomenon occurs.
 
Arc spaces provide a situation of geometric interest where non-Noetherian rings
and rings of infinite embedding dimension naturally arise. With this in mind,
we extend the definition of embedding codimension to arbitrary local rings
$(A,\fm,k)$, by setting
\[
\embcodim(A) := \height(\ker(\g))
\]
where $\g \colon \Sym_k(\fm/\fm^2) \to \gr(A)$ is the natural homomorphism. 
Geometrically, we may think of $\embcodim(A)$ as the codimension of the tangent
cone of $A$ inside its tangent space. Note that when $A$ is Noetherian this
notion agrees with the classical definition of embedding codimension, as in
this case $\dim(A)=\dim(\gr(A))$.
When $A$ is a $k$-algebra, we 
have $\embcodim(A) = 0$ if and only if $A$ is formally smooth over $k$, and
therefore one can view the embedding codimension as a (rough) measure of singularity.

If $(A,\fm,k)$ is equicharacteristic, then an alternative definition can be
given by considering the infimum of $\height(\ker(\t))$ for all surjective continuous
$\t \colon k[[x_i \mid i \in I ]] \to \^A$. We call the resulting notion
\emph{formal embedding codimension}, and denote it by $\fembcodim(A)$. 
In this paper we establish the following comparison theorem.

\begin{introtheorem}
\label{t:intro:embcodim-vs-grcodim}
For every equicharacteristic local ring $(A,\fm,k)$, we have
\[
\embcodim(A) \le \fembcodim(A),
\]
and equality holds in the following cases:
\begin{enumerate}
\item
the ring $A$ has embedding dimension
$\embdim(A) < \infty$, or
\item
there exists a scheme $X$ of finite type over $k$ such that $A$ is isomorphic
to the local ring of the arc space of $X$ at a $k$-rational point.
\end{enumerate}
\end{introtheorem}

In order to prove \cref{t:intro:embcodim-vs-grcodim} and related results on the
formal embedding codimension,
we make use of various results concerning 
infinite-variate power series rings and their localizations
which are of independent interest in the study of non-Noetherian rings.

Local rings of arc spaces and their completions where studied in
\cite{GK00,Dri,Reg06,Reg09,Reg18,MR18,dFD}. The paper \cite{dFD} looks
at the embedding dimension of the local rings to characterize stable points of
arc spaces, which were originally studied in \cite{DL99,Reg06}. 
In this paper, we consider the embedding codimension.

Let $X$ be a scheme of finite type over a field $k$, and let $X_\infty$ be its
arc space. A point $\a \in X_\infty$ corresponds to a formal arc $\a \colon
\Spec K[[t]] \to X$ where $K$ is the residue field, and defines a valuation
$\ord_\a$ on the local ring of $X$ at the base point $\a(0)$ of the arc (the
image of the closed point of $\Spec K[[t]]$). It is convenient to denote by
$\a(\e)$ the image of the generic point of $\Spec K[[t]]$. With this notation,
we can state our next theorem. 

\begin{introtheorem}
\label{t:intro:arc-finite-grcodim-embcodim}
Let $X$ be a scheme of finite type over a field $k$, and $\a \in X_\infty$.
\begin{enumerate}
\item
Assume that either $k$ has characteristic zero or $\a \in X_\infty(k)$. Then 
we have that $\embcodim(\O_{X_\infty,\a}) < \infty$ if and only if $\a(\e) \in
X_\sm$.
\item
Assume that $k$ is perfect and $\a(\e) \in X_\sm$, and 
let $X^0 \subset X$ be the irreducible component containing $\a(\e)$. Then
\[
    \embcodim(\O_{X_\infty,\a}) \le \ord_\a(\Jac_{X^0}),
\]
where $\Jac_{X^0}$ is the Jacobian ideal of $X^0$. 
\end{enumerate}
\end{introtheorem} 
  
One of the motivations behind this result comes from the following theorem,
originally conjectured by Drinfeld and proved by Grinberg and Kazhdan in
characteristic zero and then by Drinfeld in arbitrary characteristic. 
Here and in the following, we exclude the trivial case where
$X$ is, locally, just a reduced point. 

\begin{theorem}[\cite{GK00,Dri,Dri18}]
\label{th:DGK-intro}
Let $X$ be a scheme of finite type over a field $k$, and let $\a \in
X_\infty(k)$ be a $k$-rational point. If $\a(\e) \in X_\sm$, then there exists a
decomposition
\[
\^{X_{\infty,\a}} \isom \^{Z_z} \hat\times \D^\N
\]
where $\^{Z_z}$ is the formal completion of a scheme $Z$ of finite type over
$k$ at a point $z \in Z(k)$, $\D^\N = \Spf (k[[x_i \mid i \in \N]])$, 
and $\hat\times$ denotes the product in the category of formal $k$-schemes. 
\end{theorem}

Given the existence of an isomorphism as in \cref{th:DGK-intro}, we say that
$X_\infty$ admits a \emph{DGK decomposition} at $\a$. The germ $(Z,z)$ (given
by $\Spec \O_{Z,z}$) as well as its completion $\^{Z_z}=\Spf(\^{\O_{Z_z}})$ are
often referred to as a \emph{formal model} for $\a$.
Drinfeld's proof yields an algorithm for computing such a model; we 
will refer to it as a \emph{Drinfeld model}.

Partial converses of \cref{th:DGK-intro} were obtained in \cite{BS17a}, where
it is given an explicit example of a $k$-valued constant arc through the
singular locus of $X$ for which a DGK decomposition does not exist, and in
\cite{CH}, where it is shown that, in characteristic zero, if $\a$ is any
constant arc contained in the singular locus of $X$ then there are no smooth
factors in $\^{X_{\infty,\a}}$ at all. Examples of non-constant arcs that are
contained in the singular locus for which there is no DGK decomposition can
easily be constructed from these results; see also \cite{Seb16,BS17c} for
related results. An extension of the theorem to formal schemes is given in
\cite{BS17b}. Formal models of toric singularities are studied in 
\cite{BS19b}. 

As an application of our methods, we give a sharp converse to \cref{th:DGK-intro}
and provide an optimal bound to the embedded codimension of the formal model.

\begin{introtheorem}
\label{t:intro:DGK-converse}
Let $X$ be a scheme of finite type over a field $k$, and let $\a \in
X_\infty(k)$ be a $k$-rational point. 
\begin{enumerate}
\item
If $X_\infty$ admits a DGK decomposition at $\a$, then $\a(\e) \in X_\sm$. 
\item
Assume that $k$ is perfect and $\a(\e) \in X_\sm$, and let $X^0 \subset X$ 
be the irreducible component containing $\a(\e)$. 
Then for any formal model $(Z,z)$ for $\a$ we have
\[
\embcodim(\O_{Z,z}) \le \ord_\a(\Jac_{X^0}).
\]
Moreover, for every $e \in \N$ there exist $X$ and $\a$ such that both sides in
this formula are equal to $e$.
\end{enumerate}
\end{introtheorem}

The next result, which combines results of this paper with
\cref{th:DGK-intro}, provides a geometrically meaningful way of realizing a DGK
decomposition and gives an explicit formula for the embedding dimension of
a Drinfeld model.

\begin{introtheorem}
\label{th:DGK-proj-intro}
Let $X$ be an affine scheme of finite type over a perfect field $k$, let $\a \in
X_\infty(k)$ be a $k$-rational point with $\a(\e) \in X_\sm$, and let $d =
\dim_{\a(\e)}(X)$. Let $f \colon X \to Y := \A^d$ be a general linear
projection. 
\begin{enumerate} 
\item
\label{item1:DGK-proj-intro}
The map $f_\infty \colon X_\infty \to Y_\infty$ induces an isomorphism from the
Zariski tangent space of $X_\infty$ at $\a$ to the Zariski tangent space of
$Y_\infty$ at $\b := f_\infty(\a)$, and hence a closed embedding 
\[
\^{f_{\infty,\a}} \colon \^{X_{\infty,\a}} \inj \^{Y_{\infty,\b}}.
\]
\item
\label{item2:DGK-proj-intro}
For a suitable isomorphism $\^{Y_{\infty,\b}} \isom \Spec k[[u_i \mid i \in
\N]]$, the scheme $\^{X_{\infty,\a}}$ is defined in this embedding by finitely
many polynomials in the variables $u_j$, and hence the embedding gives a DGK
decomposition of $X_\infty$ at $\a$.
\item
\label{item3:DGK-proj-intro}
Let $X^0 \subset $X be the irreducible component containing $\a(\e)$,
and set $e := \ord_\a(\Jac_{X^0})$. Denote by $\^{Y_{2e-1,\b_{2e-1}}}$ the
completion of the 
$(2e-1)$-jet scheme of $Y$ at the truncation of $\b$. 
If $(Z,z)$ is a Drinfeld model compatible with the projection $f$, then the
composition of maps
\[
\^{Z_z} \inj \^{X_{\infty,\a}} \inj \^{Y_{\infty,\b}} \surj \^{Y_{2e-1,\b_{2e-1}}}
\]
gives an embedding of $\^{Z_z}$ into $\^{Y_{2e-1,\b_{2e-1}}}$,
and this embedding induces an isomorphism at the level of continuous tangent
spaces. In particular, the local ring $\O_{Z,z}$ has embedding dimension
\[
\embdim(\O_{Z,z}) = 2d \ord_\a(\Jac_{X^0}).
\]
\end{enumerate}
\end{introtheorem}

By combining \cref{t:intro:DGK-converse,th:DGK-proj-intro}, one sees that all
Drinfeld models $(Z,z)$ have the same dimension, and this dimension
satisfies
\[
(2d-1) \ord_\a(\Jac_{X^0}) \le \dim(Z) \le 2d \ord_\a(\Jac_{X^0}).
\]
In general, Drinfeld models are different from the minimal formal model. The
above theorem
implies that $2d \ord_\a(\Jac_{X^0})$ is an upper bound
on the embedding dimension of the minimal formal model.

The proofs of \cref{t:intro:arc-finite-grcodim-embcodim,th:DGK-proj-intro} rely on
a formula on the sheaf of differentials of $X_\infty$ from \cite{dFD}. A result
related to part~\eqref{item1:DGK-proj-intro} 
of \cref{th:DGK-proj-intro}, dealing with the case where $f
\colon X \to Y$ is a generically finite morphism of equidimensional schemes
with $X$ smooth, was obtained in \cite{EM09} using a more direct computation of
the map induced at the level of Zariski tangent spaces.
General projections to $\A^d$ are also used in \cite{Dri} in order to set
up the proof for the Weierstrass preparation theorem; however, Drinfeld's proof
does not lead to the same conclusions about $\^{f_{\infty,\a}}$ or about the
embedding of $\^{Z_z}$ into $\^{Y_{2e-1,\b_{2e-1}}}$. 
General projections to $\A^d$ were also used in \cite{Reg18,MR18}, 
and in fact our results give a new proof of one of the theorems of \cite{MR18}. 

There have been several attempts at extending \cref{th:DGK-intro} to a more
global statement, see \cite{BNS16,Ngo17,BK17,HW} (see also the more recent
\cite{Bou} which supersedes \cite{BK17}),
which at their core all rely
on the Weierstrass preparation theorem. The question stems from the expectation
that there should exist a well-behaved theory of perverse sheaves on arc spaces
(as well as on other closely related infinite dimensional spaces).
\cref{th:DGK-intro} suggests that one could try to define such perverse sheaves
in terms of the intersection complexes of the formal models, but one needs a
more global version of the DGK decomposition to make sense out of this. We
refer to the above citations for the motivations behind these questions.

Our interest in \cref{th:DGK-proj-intro} comes from the observation that the
same projection $f \colon X \to \A^d$ works for all arcs $\a' \in
X_\infty(k)$ in a neighborhood of $\a$ and having the same order of contact
with $\Jac_{X^0}$. The order of
contact with the Jacobian ideals of the irreducible components of $X$ gives 
a stratification of $X_\infty$, and the hope is
that the theorem may turn out to be useful in order to understand how the DGK
decomposition varies along strata. 

The paper is organized as follows. In the first few sections we review
some basic properties of power series rings in an arbitrary number of
indeterminates and establish various properties that we have been unable to
locate in the literature. Ideals of finite definition, which provide the
algebraic interpretation of a DGK decomposition, are discussed in
\cref{s:findef}. These sections provide some general results on non-Noetherian
rings and are independent from our applications to the study of singularities
of arc spaces.
In the following two sections the embedding
codimension and its formal counterpart are defined and general properties of these
notions are studied. Starting from \cref{s:genproj}, we focus on the case of arc spaces,
proving some technical theorems in \cref{s:genproj} and then addressing the
theorem of Drinfeld, Grinberg, and Kazhdan in \cref{s:DGK,s:Drinfeld-model}.
The last section is devoted to some applications related to Mather--Jacobian discrepancies,
among others to the case of toric singularities using results of \cite{BS19b}.

\cref{t:intro:embcodim-vs-grcodim} follows from
\cref{t:grcodim-embcodim,p:grcodim-embcodim:finite-embdim,c:emb=femb}.
\cref{t:intro:arc-finite-grcodim-embcodim} follows from
\cref{c:arc-finite-grcodim-embcodim,t:arc-finite-grcodim}.
\cref{t:intro:DGK-converse} follows from
\cref{t:DGK+converse,e:second-DGK-example}.
\cref{th:DGK-proj-intro} follows from
\cref{c:proj-eff-formal-emb,p:dgk-comparison}.

\subsubsection*{Acknowledgments}
We wish to thank Herwig Hauser,
Mel Hochster,
Hiraku Kawanoue,
Fran\c cois Loeser, 
Mircea Musta\c t\u a,
and Karl Schwede
for useful comments and discussions. 
We are also very grateful to the referee for their careful reading of the paper
and their valuable comments and remarks.


\section{Rings of formal power series}

\label{s:power-series}

In the paper we will work with rings of power series in an arbitrary number of
indeterminates. For our purposes, we adopt the definition of these rings as
completions of polynomial rings, a definition that differs from other standard
approaches to the theory. In this section, we briefly review the notions of
completion, graded ring, and rings of power series. All of the material here is
standard, but we want to fix notation and bring attention to some of the
subtleties that appear in the infinite-dimensional setting. 

Let $A$ be a ring and $\fm \subset A$ an ideal. We do not assume that $\fm$ is
a maximal ideal. We denote by $\^A := \liminv_n A/\fm^n$ the $\fm$-adic
completion of $A$ and regard it as a topological ring with respect to its limit
topology. Given an ideal $\fa \subset A$, we denote by $\^\fa \subset \^A$ the
completion of $\fa$ as a topological $A$-submodule. A basis for all neighborhoods
$U\subset \^A$ of $0$ is given by the descending chain of ideals
\[
  \^{\fm^n}=\ker(\^A\to A/\fm^n).
\]
The ideal $\^\fa$ then coincides with the topological closure of $\fa$ inside
$\^A$, that is,
\[
  \^\fa=\bigcap_n (\fa+\^{\fm^n}).
\]
Note that, if $\fm$ is not finitely generated, then the natural topology on $\^A$
may differ from the $\^\fm$-adic topology of $\^A$, see
\cref{r:limit-top-not-adic}.

We will denote by $\gr_{\fm}(A) := \bigoplus_{n \ge 0} \fm^n/\fm^{n+1}$ the
graded algebra of $A$ with respect to the $\fm$-adic filtration. If $\fm$ is
understood from context, we simply write $\gr(A)$ for $\gr_\fm(A)$. We will
regard $\^A$ as endowed with the filtration $\{\^{\fm^n}\}$ induced by the
completion, and therefore its graded algebra is given by $\gr(\^A) :=
\bigoplus_{n \ge 0} \^{\fm^n}/\^{\fm^{n+1}}$. There are natural isomorphisms
$\^{\fm^p}/\^{\fm^q} \isom \fm^p/\fm^q$ for all $p < q$. This gives a natural
identification between $\gr(\^A)$ and $\gr(A)$. If $\fa \subset A$ is an ideal
of $A$, then we write $\ini(\fa) := \bigoplus_{n \geq 0} (\fa \cap \fm^n)/\fm^{n+1}$
for the corresponding initial ideal. Similarly, for $\fa \subset \^A$, we set
$\ini(\fa) := \bigoplus_{n \geq 0} (\fa \cap \^{\fm^n})/\^{\fm^{n+1}}$. For an
element $f \in A$ (or $f \in \^A$), we write $\ini(f) \in \gr(A)$ for the
corresponding initial form.

Let $S$ be a ring, and let $\{x_i \mid i\in I\}$ be a collection of
indeterminates indexed by an arbitrary set $I$. We consider the polynomial ring
$P = S[x_i \mid i\in I]$ and denote by $\^P = S[[ x_i \mid i \in I]]$ the
completion of $P$ with respect to the ideal $(x_i \mid i \in I)$, that is,
\[
  S[[x_i \mid i \in I]] := \invlim_n S[x_i \mid i \in I]/(x_i \mid i\in I)^n.
\]

\begin{definition}
  We call $S[[x_i \mid i \in I]]$ the \emph{power series ring} in the
  indeterminates $x_i$, with $i \in I$, and with coefficients in $S$.
\end{definition}

\begin{remark}
  The polynomial ring $S[x_i \mid i\in I]$ is always the colimit (i.e., the
  union) of all $S[x_j \mid j \in J]$ with $J \subset I$ finite. On the other
  hand, if $I$ infinite then $S[[x_i \mid i \in I]]$ is not the colimit of all
  $S[[x_j \mid j \in J]]$ with $J \subset I$ finite; it is, however, the colimit
  of all $S[[x_j \mid j \in J]]$ with $J \subset I$ countable.
\end{remark}

\begin{remark}
  \label{r:limit-top-not-adic}
  Denoting $\fm = (x_i \mid i\in I)$, we have the exact sequence
  \[
    0 \to \^{\fm^n} \to S[[x_i \mid i \in I]] \to S[x_i \mid i \in I]/\fm^n \to 0.
  \]
  If $I$ is infinite then $S[[x_i \mid i \in I]]$ is not $\^\fm$-adically complete, i.e., 
  the natural topology on $S[[x_i \mid i \in I]]$ coming from the completion does not coincide with the $\^\fm$-adic
  one, as for instance the inclusion $\^\fm^2 \subset \^{\fm^2}$ is strict in this case, 
  see \cite[\href{https://stacks.math.columbia.edu/tag/05JA}{Tag 05JA}]{stacks-project}. 
%
\end{remark}

\begin{remark}
  \label{r:bourbaki-power-series}
  Let us contrast the above definition of $S[[x_i \mid i\in I]]$ with the ring of
  formal power series defined in \cite[Chapter~III, Section~2.11]{Bou74}, 
  which we want to briefly
  recall. For any set $I$ we denote by $\N^{(I)}$ the set of functions $I \to \N$
  that take only finitely many non-zero values. Then $\N^{(I)}$ is a monoid,
  which we identify with the collection of monomials in the variables $\{x_i \mid
  i \in I\}$ by writing 
  $x^\a = \prod_{i\in I,\:\a(i)\neq 0} x_i^{\a(i)}$ for every $\a \in \N^{(I)}$.
  The $S$-module $S^{\N^{(I)}}$ can be made into an $S$-algebra as follows:
  writing an element $a=(a_\a)_{\a\in\N^{(I)}}$ as
  $a=\sum_{\a\in\N^{(I)}} a_\a x^\a$,
  multiplication is defined via formal extension of $x^\a \cdot x^\b:=x^{\a+\b}$.
  We call $S^{\N^{(I)}}$ the \emph{ring of Bourbaki power series}.

  Notice that there is a natural inclusion of rings $S[[x_i \mid i\in I]] \subset
  S^{\N^{(I)}}$. This inclusion is an equality if $|I| < \infty$, and is a strict
  inclusion if $|I| = \infty$ as in this case $\sum_{i\in I} x_i$ is in
  $S^{\N^{(I)}}$ but not in $S[[x_i \mid i\in I]]$.
\end{remark}

\begin{remark}
  \label{r:series-of-series}
  It is often convenient to expand a formal power series in a subset of the
  indeterminates, but this becomes delicate in the infinite-dimensional case. Let
  $I$ and $J$ be arbitrary sets, and let $x_i$ and $y_j$ be indeterminates indexed by
  $i \in I$ and $j \in J$, respectively.
  Dropping for short the index sets from the notation, we
  have the following natural injections:
  \[
    S[[x_i]] \otimes_S S[[y_j]]
    \hookrightarrow
    S[[x_i]][[y_j]]
    \hookrightarrow
    S[[x_i,y_j]]
    \hookrightarrow
    (S[[x_i]])^{\N^{(J)}}.
  \]
  The first inclusion is always strict, and the other two are equalities if and
  only if $J$ is finite. For example, 
  if $x = x_{i_0}$ and $y = y_{i_0}$ are two respective indeterminates, then
  the series $\sum_{n \ge 0}x^ny^n$ belongs to $S[[x_i]][[y_j]]$ but is not in the 
  image of $S[[x_i]] \otimes_S S[[y_j]]$, and   
  if $\N \subset J$ and $x = x_{i_0}$ is one
  of the indeterminates, then the series $\sum_{n \geq 1} y_n x^n$ belongs to
  $S[[x_i, y_j]]$ but not to $S[[x_i]][[y_j]]$. Notice that Bourbaki power series
  are better behaved in this respect, as
  $S^{\N^{(I\sqcup J)}}
  =
  (S^{\N^{(I)}})^{\N^{(J)}}
  =
  (S^{\N^{(J)}})^{\N^{(I)}}$.
\end{remark}

\begin{remark}
  \label{r:gr-are-poly}
  Let $I$ be an arbitrary set, possibly infinite. We have natural
  identifications
  \[
    \gr(S[[x_i \mid i \in I]])
    \isom
    \gr(S[x_i \mid i \in I])
    \isom
    S[x_i \mid i \in I].
  \]
  Any non-zero power series $f \in S[[x_i \mid i \in I]]$ can be written as $f =
  \sum_{n=n_0}^\infty f_n$ where $f_n \in S[x_i \mid i \in I]$ is homogeneous of degree $n$
  and $f_{n_0} \neq 0$. Under the above identification, the initial form of $f$
  is given by $\ini(f) = f_{n_0}$. If $\fa \subset S[[x_i \mid i \in I]]$ is an ideal, then
  $\ini(\fa)$ gets identified with the ideal of $S[x_i \mid i \in I]$ generated by the initial
  forms of elements of $\fa$.
\end{remark}

\begin{proposition}
  \label{p:f.g.initial}
  Let $P = S[x_i \mid i \in I]$ and $\^P = S[[x_i \mid i \in I]]$,
  where $S$ is a ring and $I$ a set.
  Let $\fa \subset \^P$ be an ideal such that $\ini(\fa) \subset P$
  is finitely generated. Then $\fa$ is finitely generated and is closed in $\^P$.
\end{proposition}

\begin{proof}
  This is proven in \cite[Proposition 7.12]{Eis95} for any ring $R$ which is complete with respect to some filtration $\fm_i$.
\end{proof}

\begin{question}
\label{q:in-fin-gen}
Does the converse of \cref{p:f.g.initial} hold, that is, is the initial ideal
$\ini(\fa)$ finitely generated for any finitely generated ideal $\fa$ of $\^P$?
Alternatively, is any finitely generated $\fa$ already closed inside $\^P$?
\end{question}


\section{Embedding dimension}

\label{s:embdim}

In this section we briefly recall the notion of embedding dimension
and review some basic properties.

\begin{definition}
  The \emph{embedding dimension} of a local ring $(A, \fm, k)$ is defined to be
  \[
    \embdim(A) := \dim_k(\fm/\fm^2).
  \]
  The $k$-vector space $\fm/\fm^2$ is called the \emph{Zariski cotangent space}
  of $A$.
\end{definition}

When the local ring is equicharacteristic, the embedding dimension can
equivalently be computed as the dimension of an embedding of the completion in
a formal power series ring. Even more, if $A$ is essentially of finite type
over an infinite field $k$, then this embedding exists already at a
Zariski-local level (see \cref{t:min-zariski-emb}). Before we review these
facts, it is convenient to introduce some terminology and discuss some general
properties.

\begin{definition}
  Let $(A, \fm, k)$ be an equicharacteristic local ring. A \emph{formal
  coefficient field} of $A$ is a subfield $K \subset \^A$ that maps
  isomorphically to $\^A/\^\fm$ via the residue map.
\end{definition}

As it is well known, every equicharacteristic local ring $(A, \fm, k)$ admits a
formal coefficient field $K \subset \^A$ (see \cref{r:existence-coeff-field}).

In order to talk about a well-behaved notion of cotangent map between
completions of non-Noetherian rings, we make the following definition.

\begin{definition}
  \label{d:cont-Zariski-cotang-sp}
  Let $(A, \fm, k)$ be a local ring. The $k$-vector space $\^\fm/\^{\fm^2}$ is
  called the \emph{continuous Zariski cotangent space} of $\^A$. A collection of
  elements $a_i\in \^A$, $i\in I$, are called \emph{formal coordinates} if 
  their images $\ov a_i$ in $\^\fm/\^{\fm^2}$ form a basis.
\end{definition}

\begin{remark}
  The continuous Zariski cotangent space $\^\fm/\^{\fm^2}$ of $\^A$ is
  naturally isomorphic to the Zariski cotangent space $\fm/\fm^2$ of $A$, but
  in general it is not the same as the Zariski cotangent space $\^\fm/\^\fm^2$
  of $\^A$, as seen in \cref{r:limit-top-not-adic}.
\end{remark}

\begin{remark}
  If $(A, \fm, k)$ is a local ring admitting a coefficient field, then the
  continuous cotangent space of $\^A$ is isomorphic to $\^\Om_{A/k}
  \otimes_{\^A} k$, where
  \[
   \^\Om_{A/k}:=\varprojlim_n \Om_{(A/\fm^n)/k}
  \]
  is defined as in \cite[Chapter~$0_{\text{IV}}$, 20.7]{EGAiv_1}.
\end{remark}

\begin{definition}
  \label{def:formal-emb}
  Let $(A, \fm, k)$ be an equicharacteristic local ring. A \emph{formal
  embedding} of $A$ is a surjective continuous homomorphism $\t \colon \^P \to
  \^A$ where $\^P = k[[x_i \mid i\in I]]$ is a power series ring. A formal
  embedding $\t$ is called \emph{efficient} if the induced map at the level of
  continuous Zariski cotangent spaces $\^\fn/\^{\fn^2} \to \^\fm/\^{\fm^2}$ is
  an isomorphism.
\end{definition}

\begin{proposition}
  \label{l:tau-exists}
  Let $(A, \fm, k)$ be an equicharacteristic local ring. Let $K \subset \^A$ be
  a formal coefficient field, and let $a_i\in \^A$, $i\in I$, be formal
  coordinates. Then there exists a unique efficient formal embedding $\t \colon
  \^P=k[[ x_i \mid i\in I]] \to \^A$ such that $\t(k) = K$ and $\tau(x_i) =
  a_i$. Every efficient formal embedding of $A$ is of this form.
\end{proposition}

\begin{proof}
  First, note that for every $n \ge 1$ the composition $K \to \^A \to
  \^A/\^{\fm^n}$ is injective (since $K$ maps isomorphically to the residue
  field $k = \^A/\^\fm$) and hence gives an embedding $K \subset A/\fm^n$ via
  the natural isomorphism $\^A/\^{\fm^n} \isom A/\fm^n$. Letting $P = k[x_i
  \mid i \in I]$ and $\fn = (x_i \mid i \in I) \subset P$, we have compatible
  homomorphisms $\t_n \colon P/\fn^n \to A/\fm^n$ such that $\t_n(k) = K$ and
  $\t_n(x_i) = a_i+\fm^n$. Taking limits, these maps define $\t$ and determine
  it uniquely. By construction $\gr(\t)$ is surjective, and hence $\t$ is
  surjective by \cite[Chapter~III, Section~2.8, Corollary~2]{Bou72}. Since
  $\t$ induces an isomorphism at the level of continuous Zariski cotangent
  spaces, we see that $\t$ is an efficient formal embedding. For the last
  statement, notice that if $\t$ is an efficient formal embedding, then clearly
  $K = \t(k)$ is a formal coefficient field and $\t(x_i)$, $i\in I$, are formal
  coordinates.
\end{proof}

The map $\t$ in \cref{l:tau-exists} can be interpreted as follows. For short,
let $P := \Sym_k(\fm/{\fm^2})$. Every choice of formal coefficient field $K
\subset \^A$ and formal coordinates $a_i \in \^A,$ $i \in I$, determines an
embedding $\fm/{\fm^2} \inj \^A$ as a $K$-vector space, and hence a map $\t_0
\colon P \to \^A$. Letting $x_i = a_i + \fm^2 \in P$, we get a natural
identification $P = k[x_i \mid i\in I]$. Then the map $\t$ is obtained from
$\t_0$ by completing the domain $P$.

\begin{remark}
  \label{r:tau-using-Sym}
  There is also a natural embedding $\fm/\fm^2 \inj \gr(A)$ as a $k$-vector
  space, which induces a map $\g \colon P \to \gr(A)$. It is immediate from the
  construction that $\gr(\t) = \g$. In particular we see that $\gr(\t)$ is
  independent of any choices. On the other hand, $\t$ itself certainly depends
  on the choices of the formal coefficient field $K$ and formal coordinates
  $a_i\in\^A$, $i\in I$. 
\end{remark}

The following result addresses the dependence of $\t$ on $K$ and $\{a_i \mid
i\in I\}$.

\begin{proposition}
  \label{p:emb-indep-basis}
  Let $(A, \fm, k)$ be an equicharacteristic local ring, let $K, K' \subset
  \^A$ be two formal coefficient fields, and let $\{a_i \mid i\in I\} \subset
  \^A$ and $\{a'_i \mid i\in I\} \subset \^A$ be two sets of formal
  coordinates. Consider the two maps $\t \colon \^P := k[[x_i \mid i\in I]] \to
  \^A$ and $\t' \colon \^P' := k[[x_i' \mid i\in I]] \to \^A$ given by
  \cref{l:tau-exists}. Then there exists an isomorphism $\f \colon \^P' \to
  \^P$ such that $\t' = \t \o \f$. 
\end{proposition}

The proof of \cref{p:emb-indep-basis} relies on the following basic property of
formally smooth algebras. The statement is a natural generalization of the
definition of formal smoothness which guarantees lifting not only via
extensions with nilpotent kernel, but also via extensions with topologically
nilpotent kernel.

\begin{proposition}
  \label{p:formally-smooth-lift}
  Let $k_0$ be a topological ring, and $k$ a formally smooth $k_0$-algebra. Let
  $C$ be a complete metrizable topological $k_0$-algebra, and $\cI \subset C$ a
  closed ideal such that $\{\cI^n\}$ tends to zero. Then every continuous
  $k_0$-algebra homomorphism $u \colon k \to C/\cI$ factors as $k
  \xrightarrow{v} C \to C/\cI$, where $v$ is a continuous $k_0$-algebra
  homomorphism.
  \[
  \xymatrix{
  k_0 \ar[r] \ar[d]
  & C \ar[d]
  \\ k \ar[r]^{u} \ar@{-->}[ur]^{v}
  & C/\cI
  }
  \]
\end{proposition}

\begin{proof}
  See \cite[Chapter~$0_{\text{IV}}$, Proposition~19.3.10]{EGAiv_1}.
\end{proof}

\begin{remark}
  \label{r:existence-coeff-field}
  \cref{p:formally-smooth-lift} implies the existence of formal coefficient
  fields for any equicharacteristic local ring $(A, \fm, k)$. In this case $C =
  \^A$, $\cI = \^\fm$, $k_0$ is the prime field contained in $\^A$, $k$ is the
  residue field, and $u$ is the identity. Notice that $k_0$ is perfect, and
  therefore $k$ is separable over $k_0$ (hence formally smooth). Then $K =
  v(k)$ is a formal coefficient field.
\end{remark}

Let $S$ be any discrete topological ring. For any two topological $S$-algebras
$T$ and $T'$ the tensor product $T\otimes_S T'$ is endowed with the final
topology with respect to its natural maps. The \emph{completed tensor product}
$T \cotimes_S T'$ is defined to be the completion of $T\otimes_S T'$. Note that
the operation $\cotimes_S$ is the coproduct in the category of complete
topological $S$-algebras.

\begin{lemma}
\label{p:inverse-fn-thm}
  Let $(S,\fn,k)$ be a local $k$-algebra. Any continuous $S$-algebra map
  \[
    \varphi
    \colon
    S \,\cotimes_k\, k[[t_i \mid i\in I]]
    \to
    S \,\cotimes_k\, k[[z_i \mid i\in I]]
  \]
  which induces an isomorphism of continuous cotangent spaces is an isomorphism.
\end{lemma}

\begin{proof}
  Note that a basis for the topology on $S \otimes_k k[[t_i \mid i \in I]]$ is given by the
  filtration
  \[
  \fm_n:=\sum_{d+e=n} \fn^d + \big((t_i \mid i\in I)^e\big)\widehat{\phantom{t}}.
  \]
  Thus it follows that for the associated graded rings we have
  \[
  \gr(S \,\cotimes_k\, k[[t_i \mid i\in I]])\isom\gr(S \otimes_k k[t_i \mid
  i\in I])\isom\gr(S)\otimes_k k[t_i \mid i\in I].
  \]
  The map $\varphi$ induces a $\gr(S)$-algebra map
  \[
  \gr(S)[t_i \mid i\in I] \to \gr(S)[z_i \mid i\in I], 
  \]
  which by assumption is an isomorphism. Thus we can use
  \cite[Lemma~10.23]{AM69} to see that $\varphi$ is bijective. It is easy to
  check that $\varphi^{-1}$ is continuous and we are done.
\end{proof}

\begin{proof}[Proof of \cref{p:emb-indep-basis}]
  Let $k_0$ be the prime field contained in $\^A$. Notice that $k$ is formally
  smooth over $k_0$. We apply \cref{p:formally-smooth-lift} in the situation in
  which $C = \^P$, $\cI = \ker(\t)$, and $u \colon k \to \^A = C/\cI$ is the
  map such that $u(k) = K'$. Notice that $\cI = \t^{-1}(0)$ is closed because
  $\^A$ is separated, and that $\{\cI^n\}$ tends to zero because $\cI^n \subset
  \^{\fn^n}$. We get a map $v \colon k \to \^P$ verifying $\t(v(k)) = K'$.

  Since $\t$ is surjective, there exist power series $f_i \in \^P$ such that
  $\t(f_i) = a_i'$. The map $\f$ is given by $\f(x_i') = f_i$ and $\f|_k = v$.
  \cref{p:inverse-fn-thm} shows that $\f$ is an isomorphism.
\end{proof}

\begin{remark}
  \label{r:emb-indep-basis}
  By the same argument of the proof of \cref{p:emb-indep-basis}, one can see
  that given any two formal embeddings $\t \colon \^P \to \^A$ and $\t' \colon
  \^P' \to \^A$ (not necessarily efficient) there is always a map $\f \colon
  \^P' \to \^P$ such that $\t' = \t \o \f$, and if $\t$ efficient then $\f$ is
  surjective. 
\end{remark}

\begin{proposition}
  \label{th:embdim=inf}
  For every equicharacteristic local ring $(A, \fm, k)$ we have
  \[
    \embdim(A) = \min_{\tau} \dim \^P
  \]
  where the minimum is taken over all choices of formal embeddings $\tau \colon
  \^P \to \^A$ and is achieved whenever $\t$ is an efficient formal embedding.
\end{proposition}

\begin{proof}
  Let $\t \colon \^P \to \^A$ be a formal embedding, and write $P = k[x_i \mid i
  \in I]$ and $\fn = (x_i \mid i\in I) \subset P$. Since $\t$ is continuous, we
  have that $\t(\^\fn^c) \subset \^\fm$ for some $c$. As $\^\fm$ is maximal, this
  forces $\t(\^\fn) \subset \^\fm$, and continuity gives $\t(\^{\fn^n}) \subset
  \^{\fm^n}$ for all $n$. Hence we get an induced map at the level of graded
  rings $\gr(\t) \colon P \to \gr(A)$. Since $\t$ is surjective, $\gr(\t)$ is
  also surjective and $\t(\^{\fn^n}) = \^{\fm^n}$ for every $n$. In particular
  $\t$ induces a surjection at the level of Zariski cotangent spaces $\fn/\fn^2
  \to \fm/\fm^2$ and we see that $\embdim(A) \leq \dim \^P$. If $\t$ is an
  efficient formal embedding, then the map $\fn/\fn^2 \to \fm/\fm^2$ is an
  isomorphism and we have $\embdim(A) = \dim \^P$. 
\end{proof}

We finish this section by recalling the following result, which guarantees the
existence of a Zariski-local minimal embedding for singular points of a scheme
of finite type over an infinite field. This is well-known in the case of
complex varieties (see for example \cite[Theorem 3]{BK84}) and we provide an
extension of the proof to the more general case considered here.

\begin{theorem}
  \label{t:min-zariski-emb}
  Let $X$ be a scheme of finite type over an infinite field $k$ and $x\in
  X(k)$. If $\embdim(X,x)=d$ and $X$ is not smooth at $x$, then there exists a
  closed subscheme $Y\subset \A^d_k$, a point $y\in Y(k)$, and an isomorphism
  \[
   \O_{Y,y} \isom \O_{X,x}.
  \]
\end{theorem}
 
\begin{proof}
  We may assume that $X$ is projective and embedded in $\P^n$ for some $n> d$.
  Denote by $\bar{k}$ the algebraic closure of $k$ and write
  $\bar{X}:=X\times_k \Spec(\bar{k})$ and $\bar{x}$ for the $\bar{k}$-point on
  $\bar{X}$ corresponding to $x$. As $\O_{\bar{X},\bar{x}}$ is not a regular
  ring, we have
  \[
   \dim_{\bar{x}}(\bar{X})<\embdim(\bar{X},\bar{x})=\embdim(X,x)=d.
  \]
  
  Suppose we can find a linear projection $\pi\colon \P^n \to \P^d$ defined
  over $k$ such that, if $\bar{Y}$ denotes the scheme-theoretic image of
  $\bar{X}$ under $\pi$ and $\bar{y}=\pi(\bar{x})$, then the induced map
  $\O_{\bar{Y},\bar{y}}\to \O_{\bar{X},\bar{x}}$ is an isomorphism. Since
  $\bar{Y}=Y\times_k \Spec(\bar{k})$, where $Y$ is the scheme-theoretic image
  of $X$ under the linear projection centered at $x$, we obtain a map
  $\O_{Y,y}\to \O_{X,x}$ whose base change to $\bar{k}$ gives the map above.
  Thus, by faithfully flat descent, we get that $\O_{Y,y}\isom \O_{X,x}$.
  
  Now, in order to prove the claim, let $\bar{T}\subset \P^n_{\bar{k}}$ be the
  unique linear space passing through $\bar{x}$ whose tangent space at
  $\bar{x}$ agrees with that of $\bar{X}$. Furthermore, let $\bar{S}$ be the
  closure of the set of all lines connecting $\bar{z}$ with $\bar{x}$, where
  $\bar{z}\in \bar{X}$, $\bar{z}\neq \bar{x}$. Note that
  $\dim(\bar{S})=\dim(\bar{X})+1\leq d$. Consider now the closure $\bar{Z}$ of
  the set $\bar{X}\cup \bar{T} \cup \bar{S}$, equipped with its reduced scheme
  structure. Since $\dim(\bar{S})\leq d$ the set of all linear spaces $\bar{L}$
  with $\bar{L}\cap \bar{Z}=\emptyset$ is open inside $\Gr(n-d-1,n)\times_k
  \Spec(\bar{k})$. The preimage of this set in $\Gr(n-d-1,n)$ is a nonempty
  open set, and since $k$ is infinite, it has a $k$-rational point, which we
  denote by $L$. Hence we have that the corresponding projection $\pi_L\colon
  \P^n\to \P^d$, defined over $k$, satisfies $\pi_{\bar{L}}^{-1}(\bar{y}) \cap
  \bar{X} =\{\bar{x}\}$ set theoretically, where $\bar{y}$ corresponds to the
  $k$-point $y:=\pi_L(x)$. Writing $\bar{Y}:=\pi_{\bar{L}}(\bar{X})$ we get
  that the map of local rings $\O_{\bar{Y},\bar{y}}\to \O_{\bar{X},\bar{x}}$ is
  injective and finite. Since $\bar{L}\cap \bar{T}=\emptyset$ the tangent
  spaces of $\bar{x}$ and $\bar{y}$ are isomorphic and thus
  $\fm_{\bar{y}}\O_{\bar{X},\bar{x}}=\fm_{\bar{x}}$. The claim now follows from
  the Nakayama lemma.
\end{proof}


\section{Flatness of completion}

\label{s:flatness}

Let $A$ be a ring and $\fm$ an ideal in $A$. Given an $A$-module $E$ we will
consider the $\fm$-adic topology on $E$ and we will denote by $\^E$ its
$\fm$-adic completion. We are interested in conditions guaranteeing that
the natural map $A \to \^A$ is flat.  

\begin{definition}
  Let $E$ be an $A$-module and $F$ a submodule of $E$. We say that $F \subseteq
  E$ has the \emph{Artin--Rees property} with respect to $\fm$ if there exists a
  $c \in \N$ such that, for all $n > c$, we have
  \[
    \fm^n E \cap F = \fm^{n-c}(\fm^c E \cap F).
  \]
  The smallest such $c$ is called the \emph{Artin--Rees index} of $F \subseteq E$
  with respect to $\fm$. We say that $A$ has the \emph{Artin--Rees property} with
  respect to $\fm$ if so does every finitely generated submodule of a finitely
  generated free $A$-module.
\end{definition}

The Artin--Rees property for $F \subseteq E$ guarantees that the $\fm$-adic
topology of $F$ coincides with the topology induced by the $\fm$-adic topology
of $E$. In this context it is natural to consider the Rees algebra $A^* =
\bigoplus_{n \geq 0} \fm^n$ and the graded $A^*$-modules
\[
  E^* = \bigoplus_{n \geq 0} \fm^n E
  \qquad\text{and}\qquad
  F^* = \bigoplus_{n \geq 0} \fm^n E \cap F.
\]

\begin{lemma}
  \label{l:rees-algebra}
  $F \subseteq E$ has the Artin--Rees property if and only if there exists a $c
  \in \N$ such that $F^*$ is generated as a graded $A^*$-module by elements of
  degree $\leq c$. Moreover, the Artin--Rees index of $F \subseteq E$ is the
  smallest such $c$.
\end{lemma}

\begin{proof}
  This is immediate from the definitions. Compare with \cite[Chapter~III,
  Section~3.1, Theorem~1]{Bou72} or \cite[Theorem~8.5]{Mat89} or
  \cite[Lemma~10.8]{AM69}, but notice that no finite generation hypotheses are
  needed for the statement of the lemma.
\end{proof}

\begin{remark}
\label{r:example-AR-fails}
  By the classical Artin--Rees lemma \cite[Theorem~8.5]{Mat89}, any Noetherian
  ring $A$ has the Artin--Rees property with respect to any ideal $\fm \subset
  A$. By contrast, there exist non-Noetherian rings, even finite dimensional,
  which do not have the Artin--Rees property. A zero-dimensional example is given
  by 
  \[
    A = k[x_i \mid i \in \N]/(x_1-x_m^m \mid m \ge 2) + (x_n^{n+1} \mid n \ge 1),
  \] 
  with $\fm = (x_i \mid i \in \N)$ and $F = (x_1) \subset E = A$. Clearly $x_1
  \in \fm^n$ for all $n$, but there is no $f \in \fm$ such that $x_1 = x_1 f$.
\end{remark}

In complete analogy with the Noetherian case, we prove that the Artin--Rees
property implies flatness of the completion. We recall that a ring is
\emph{coherent} if every finitely generated ideal is finitely presented.

\begin{proposition}
  \label{p:flatness-completion}
  Let $A$ be a coherent ring with the Artin--Rees property with respect to $\fm
  \subset A$, and let $\^A$ be its $\fm$-adic completion. Then $A \to \^A$ is
  flat. Moreover, if $\fa \subset A$ is a finitely generated ideal, then $\fa
  \^A$ is closed in $\^A$ (that is, $\fa \^A = \^\fa$).
\end{proposition}

\begin{proof}
  Let $\fa$ be a finitely generated ideal of $A$. Since $A$ is coherent, there
  exists an exact sequence
  \[
    \xymatrix{
      A^p \ar[r] & A^q \ar[r]^\varphi & \fa \ar[r] & 0.
    }
  \]
  Moreover, since the Artin--Rees property holds for $\ker \varphi \subset A^q$,
  the $\fm$-adic topology on $\ker \varphi$ agrees with the one induced by the
  inclusion $\ker \varphi \subset A^q$. From \cite[Chapter~III, Section~2.12,
  Lemma~2]{Bou72} or \cite[Lemma~10.3]{AM69}, the sequence remains exact after
  taking $\fm$-adic completions,
  and we have a commutative diagram
  \[
    \xymatrix@C=10pt{
      A^p \otimes_A \^A \ar[r] \ar[d] & A^q \otimes_A \^A \ar[r] \ar[d] 
      & \fa \otimes_A \^A \ar[r] \ar[d] & 0 \\
      \^{A^p} \ar[r] & \^{A^q} \ar[r]^{\^\varphi} & \^\fa \ar[r] & 0
    }
  \]
  with exact rows. Since taking completion commutes with finite direct sums, the
  map $\fa \otimes_A \^A \to \^\fa$ is an isomorphism. As the natural map $\^\fa
  \to \^A$ is an injection, flatness of $A \to \^A$ follows from
  \cite[Theorem~7.7]{Mat89}. The fact that $\fa \otimes_A \^A \to \^\fa$ is an
  isomorphism also shows that $\fa\^A = \^\fa$. 
\end{proof}

The following theorem gives a first example of a non-Noetherian ring with the
Artin--Rees property. We were not able to find a reference for this
statement in the literature.

\begin{theorem}
  \label{t:AR-property}
  Let $S$ be a Noetherian ring and $\fn$ any ideal of $S$. For any set $I$
  consider $P = S[x_i \mid i \in I]$ and $\fm = (x_i \mid i \in I) + \fn$. Then
  $P$ has the Artin--Rees property with respect to $\fm$.
\end{theorem}

\begin{proof}
  Let $E$ be a finitely generated free $P$-module and $F \subseteq  E$ a finitely
  generated submodule. Assume that $E$ is freely generated by $e_1, \ldots, e_s$

  Given any subset $J \subseteq I$, we write $P_J := S[x_i \mid i \in J]$, and for
  any ideal $\fa \subseteq P$ we denote $\fa_J := \fa \cap P_J$. We define $E_J :=
  P_J \cdot e_1 \oplus \cdots \oplus P_J \cdot e_s$, and for any $P$-submodule $G
  \subseteq E$ we write $G_J := E_J \cap G$. Note that $P, \fm, \fa, E, G$ are the
  colimits of $P_J, \fm_J, \fa_J, E_J, G_J$ for $J \subseteq I$ finite. 
  We have
  \[
      G_J \cap G'_J
      =
      (G \cap G')_J,
      \quad
      \fa_J G_J \subseteq (\fa G)_J,
      \quad
      \fa_J E_J
      = 
      (\fa E)_J,
      \quad\text{and}\quad
      (\fm_J)^n
      =
      (\fm^n)_J.
  \]
  In particular, for
  all $n, d \in \N$ with $n > d$, we have
  \[
      \fm_J^n E_J \cap F_J
      =
      (\fm^n E \cap F)_J
      \quad\text{and}\quad
      \fm_J^{n-d}(\fm_J^d E_J \cap F_J)
      \subseteq
      (\fm^{n-d}(\fm^d E \cap F))_J.
  \]

  Assume that $F$ is generated by $f_1, \ldots, f_r$. Then there exists a finite
  set $L \subset I$ such that $f_1,\dots,f_r \in F_L$, and for any $J$ with $L
  \subseteq J \subseteq I$ we have $F_J = P_J \cdot f_1 + \cdots + P_J \cdot f_r
  = P_J \cdot F_L$.

  Since $P_L$ is Noetherian it has the Artin--Rees property with respect to
  $\fm_L$, and hence there exists a $c \in \N$ such that
  \[
    \fm_L^n E_L \cap F_L = \fm_L^{n-c}(\fm_L^c E_L \cap F_L)
  \]
  for all $n > c$. The smallest such $c$ is the Artin--Rees index of $F_L
  \subseteq E_L$. Since for any finite set $J$ with $L \subseteq J \subset I$ we
  have $F_J = P_J \cdot F_L$, we can apply \cref{l:AR-index} and we see that the
  Artin--Rees index of $F_J \subseteq E_J$ is again $c$. This implies that
  \[
      (\fm^n E \cap F)_J
      \subseteq
      (\fm^{n-c}(\fm^c E \cap F))_J.
  \]
  Taking the colimit for all finite $J \subset I$ we get that
  \[
      \fm^n E \cap F
      \subseteq
      \fm^{n-c}(\fm^c E \cap F).
  \]
  The reversed inclusion is immediate, and the theorem follows.
\end{proof}

\begin{lemma}
  \label{l:AR-index}
  Let $A_0$ be a Noetherian ring, $\fm_0 \subset A_0$ an ideal, $E_0$ a finitely
  generated $A_0$-module, and $F_0 \subseteq E_0$ a submodule. Let $z$ be a new
  variable and consider the ring $A = A_0[z]$, the ideal $\fm = \fm_0 A + (z)$,
  the extension $E = A \otimes_{A_0} E_0 = E_0[z]$, and $F = A \otimes_{A_0} F_0
  = F_0[z]$. Then the Artin--Rees index of $F \subseteq E$ with respect to $\fm$
  equals the Artin--Rees index of $F_0 \subseteq E_0$ with respect to $\fm_0$.
\end{lemma}

\begin{proof}
  Let $c_0$ and $c$ be the Artin--Rees indexes of $F_0 \subseteq E_0$ and $F
  \subseteq E$. As in \cref{l:rees-algebra}, consider the Rees algebras
  \[
      A_0^* = \bigoplus_{n \geq 0} \fm_0^n
      \quad\text{and}\quad
      A^* = \bigoplus_{n \geq 0} \fm^n,
  \]
  and the graded modules
  \[
      F^*_0 = \bigoplus_{n \geq 0} \fm_0^n E_0 \cap F_0
      \quad\text{and}\quad
      F^* = \bigoplus_{n \geq 0} \fm^n E \cap F.
  \]
  Then $F_0$ is generated in degree $\leq c_0$ as a graded $A_0$-algebra (and not
  in any lower degree), and similarly for $F$.

  Any element $f \in \fm^n E \cap F$ can be written as
  $f = \sum_{i=0}^n f_i z^{n-i}$ where $f_i \in \fm_0^i E_0 \cap F_0$.
  In particular, $F^*$ is generated by $F^*_0$ as an $A^*$-algebra, and therefore
  $c \leq c_0$. Conversely, if $F^*$ is generated by homogeneous elements
  $f^{(1)}, \ldots, f^{(r)}$ with $f^{(j)} = \sum_i f_i^{(j)} z^{n_j -i}$, then
  $F^*_0$ is generated by $f^{(1)}_{n_1}, \ldots, f^{(r)}_{n_r}$. We see that
  $c_0 \leq c$, and the result follows.
\end{proof}

\begin{remark}
  If $A$ and $\fm$ are as in \cref{r:example-AR-fails} then we have $A =
  \dirlim_m A_m$ where 
  \[
  A_m = k[x_1,\dots,x_m]/(x_1-x_i^i,x_i^{i+1},\:1<i\leq m).
  \]
  It is easy to check that the Artin--Rees index of $(x_1) \subset A_m$ is $m$
  and $A$ does not have the Artin--Rees property.
\end{remark}

Recall that for any discrete topological ring $S$ and any two topological
$S$-algebras $T$ and $T'$, the completed tensor product $T \cotimes_S T'$ is
defined to be the completion of $T\otimes_S T'$ with respect to its natural
topology.

\begin{corollary}
  \label{c:flatness-completion}
  Let $S \to T$ be a map of Noetherian rings. As above, suppose that $S$ has the
  discrete topology, and let $T$ be equipped with the $\fn$-adic topology where
  $\fn \subset T$ is an ideal. Then the natural map
  \[
    T[x_i \mid i\in I]
    = T\otimes_S S[x_i \mid i\in I]
    \to
    T \cotimes_S S[x_i \mid i\in I]
  \]
  is flat. In particular:
  \begin{enumerate}
    \item
    \label{item1:flatness-completion}
    for any index set $I$ the completion map
    $S[x_i \mid i\in I] \to S[[x_i \mid i\in I]]$
    is flat, and
    \item
    \label{item2:flatness-completion}
    for every finite subset $J\subset I$ the inclusion
    $S[[x_j \mid j\in J]] \to S[[x_i \mid i\in I]]$
    is flat.
  \end{enumerate}
\end{corollary}

\begin{proof}
  Observe that a basis for the topology on $T[x_i \mid i\in I]$ is given by
  \[
    \fn^m[x_i \mid i\in I]+(x_i \mid i\in I)^n,\quad m,n\in\N,
  \]
  which is easily seen to be equivalent to the $\fm$-adic one, where $\fm:=(x_i
  \mid i\in I)+\fn$. As $T[x_i \mid i \in I]$ is coherent 
  (e.g., see \cite[Theorem 6.2.2]{Gla89}), 
  the first assertion follows from
  \cref{t:AR-property,p:flatness-completion}. Regarding the last two assertions,
  \eqref{item1:flatness-completion} 
  follows by observing that $S \cotimes_S S[x_i \mid i\in I] = S[[x_i \mid
  i\in I]]$, and \eqref{item2:flatness-completion} 
  by taking $T = S[[x_j \mid j\in J]]$ with the $(x_j \mid
  j\in J)$-adic topology and observing that the given inclusion factors as
  \[
    S[[x_j \mid j\in J]]
    \to S[[x_j \mid j\in J]][x_i \mid i\in I \setminus J]
    \to S[[x_i \mid i\in I]]
  \]
  and so is flat.
\end{proof}

\begin{remark}
  For quotients $A$ of $k[x_i \mid i\in I]$ the completion map $A\to \^{A}$ need
  not be flat, even if the topology of $A$ is separated. Consider the ideal
  \[
    \fa=(yx_1,yx_n^{n}-zx_{n-1}^{n-1} \mid n>1)
  \]
  in $P=k[x_n,y,z \mid n\in\N_{>0}]$ and the quotient $A=P/\fa$. Let
  $\fm=(x_n,y,z \mid n\in\N_{>0})\subset A$. As $\fm$ is weighted homogeneous
  with respect to the positive weights $w(x_n)=w(y)=1$, $w(z)=2$, it follows that
  the $\fm$-adic topology on $A$ is separated. Consider the element $y-z$, which
  is annihilated by the series $f=\sum_{n\geq1} x_n^n$. If $\^A$ were flat over
  $A$, there would exist polynomials $a_1,\ldots,a_r\in A$ annihilating $y-z$
  such that $f$ can be written as
  $f=\sum_{j=1}^r a_j b_j$ where $b_j\in \^A$.

  Considering this equation modulo $(y,z)$, we have written $f$ as a linear
  combination of polynomials in $k[x_n \mid n\in\N_{>0}]$, which is clearly
  impossible.
\end{remark}

We close this section with the following analogue to \cref{p:f.g.initial}
for polynomial rings. 

\begin{proposition}
Let $P = S[x_i \mid i \in I]$ and $\fm = (x_i \mid i \in I)$, where $S$ is a ring and $I$ a set.
Let $\fa \subset P_\fm$ be an ideal such that $\ini(\fa) \subset P$ is finitely generated. 
Then $\fa$ is finitely generated.
\end{proposition}

\begin{proof}
Let $f_1,\ldots,f_r\in \fa$ be such that $\ini(f_1),\ldots,\ini(f_r)$ generate $\ini(\fa)$. Since 
$\ini(\fa)=\ini(\^\fa)$, we can apply \cref{p:f.g.initial} to see that $\^\fa=(f_1,\ldots,f_r)\^P$. 
By \cref{c:flatness-completion} the map $P_\fm\to\^P$ is faithfully flat and thus 
$\fa\subset \^\fa\cap P_\fm=(f_1,\ldots,f_r)P_\fm$. The other inclusion is trivial, so 
$\fa=(f_1,\ldots,f_r)$.
\end{proof}


\section{Ideals of finite definition}

\label{s:findef}

In this section, we fix a field $k$ and a set $I$, and consider the polynomial
ring $P = k[x_i \mid i \in I]$ and the power series ring $\^P = k[[x_i \mid i
\in I]]$. An important class of ideals in $\^P$ are those generated by finitely
many power series involving only finitely many variables. We study their properties
in this section.

For any subset $J \subset I$, we write $P_J = k[x_i \mid i \in J]$ and $\^P_J =
k[[x_i \mid i \in J]]$, and for any ideal $\fa \subset \^P$ we denote $\fa_J :=
\fa \cap \^P_J$. 

\begin{definition}
  Let $\fa \subset \^P$ be an ideal.
  \begin{enumerate}
   \item
   We say that $\fa$ is \emph{of finite definition} with respect to the
   indeterminates $x_i$ if there exists a finite subset $J\subset I$ such that
   $\fa=\fa_J \^P$.

   \item
   Similarly, $\fa$ is \emph{of finite polynomial definition} with
   respect to the indeterminates $x_i$ if it is generated by finitely many
   polynomials, i.e., elements in $P$.

   \item
   We say that $\fa$ is \emph{of finite (polynomial) definition} if there
   exist a $k$-isomorphism $\^P \isom k[[x'_i \mid i\in I]]$ such that $\fa$ is
   of finite (polynomial) definition with respect to the formal coordinates
   $x'_i$. 
  \end{enumerate}
\end{definition}

\begin{definition}
  Let $(A, \fm, k)$ be an equicharacteristic local ring.
  \begin{enumerate}
   \item
   A \emph{weak DGK decomposition} for $A$ is an isomorphism $\^A \isom k[[x_i
   \mid i \in I]]/\fa$ where $\fa$ is an ideal of finite definition.

   \item
   A \emph{DGK decomposition} for $A$ is an isomorphism $\^A \isom k[[x_i \mid
   i\in I]]/\fa$ with $\fa$ of finite polynomial definition.

   \item
   We say that a (weak) DGK decomposition $\^A \isom k[[x_i \mid i\in I]]/\fa$
   is \emph{efficient} if the quotient map $k[[x_i \mid i\in I]] \to \^A$ is an
   efficient formal embedding.
  \end{enumerate}
\end{definition}

\begin{remark}
\label{r:DGK-and-finite-def}
  If $A$ has a DGK decomposition, then we have an isomorphism
  $\^A \isom \^B \,\cotimes_k\, \^P$
  where $\^P$ is a power series ring and $(B, \fn, k)$ is a local $k$-algebra
  which is essentially of finite type.
  Geometrically, this means that
  $\Spf(\^A) \isom \^Z_z \^\times \D^I$
  where $\D^I = \Spf (k[[x_i \mid i \in I]])$ and $\^Z_z$ is the formal
  neighborhood of a scheme $Z$ of finite type over $k$ at a point $z \in Z(k)$.
  If $A$ has a weak DGK decomposition, then 
  $\^A \isom \B \,\cotimes_k\, \^P$
  where $\B$ is a Noetherian complete local ring with residue field $k$.
\end{remark}

\begin{example}
  \label{e:weak-DGK-vs-strong-DGK}
  The existence of a weak DGK decomposition for a ring $A$  does not imply the
  existence of a DGK decomposition for $A$. This can be seen by considering the
  following example given by Whitney. Let $f(t)$ be a transcendental power
  series with complex coefficients and
  with $f(0)=0$, and consider the equation
  \[
    g = xy(y-x)(y-(3+t)x)(y-(4+f(t))x).
  \]
  It is proven in \cite[Example~14.1]{Whi65} that $\mathcal B = \C[[x,y,t]]/(g)$
  is not isomorphic to the completion
  of a local ring of a $\C$-scheme of finite type. In particular, any local
  ring $A$ for which $\^A \isom \mathcal B$ (for example, $\mathcal B$ itself)
  admits a weak DGK decomposition but not a DGK decomposition.

  We now give another example of a local ring $A$ such that $\^A \isom
  \mathcal B$. This example has the advantage of being explicitly presented as
  the localization of a quotient of a polynomial ring in countably many
  variables. Write $f(t)=\sum_{i\geq1} a_i t^i \in \C[[t]]$. Consider the
  polynomial ring $P = \C[x,y,t,z_n \mid n \geq 0]$ and the ideal  
  \[
   \fa = (h, z_{n-1} - z_n t - a_n t \mid n\geq 1)
  \]
  where
  \[
    h = xy(y-x)(y-(3+t)x)(y-(4+z_0)x).
  \]
  Let $A$ be the localization of $P/\fa$ at the ideal 
  $(x,y,t,z_n \mid n \geq 0)$.
  Then, in $\^A$, we have for each $m\geq 1$
  \[
   z_0 - f(t)
   =
   z_m t^m - \sum_{i\geq m+1} a_i t^i \in \^\fm^{m},
  \]
  and for each $m \geq n+1$
  \[
   z_n - \sum_{i\geq n+1} a_i t^{i-n}
   =
   z_m t^{m-n} - \sum_{i\geq m+1} a_i t^{i-n} \in \^\fm^{m-n},
  \]
  Thus it follows that $\^A \isom \C[[x,y,t,z_0]]/(h, z_0-f(t)) \isom
  \C[[x,y,t]]/(g) = \mathcal B$.
\end{example}

\begin{remark}
  \label{r:counterex-finite-def} 
  An analogous definition of \emph{finite definition} can be given for ideals in
  a polynomial ring $P = k[x_i \mid i \in I]$. It is easy to see that the
  definition does not depend on the choice of indeterminates, and that an ideal
  of $P$ is of finite definition if and only if it is finitely generated. By
  contrast, in a power series ring not every ideal of finite definition is so
  with respect to the given indeterminates $x_i$, and not every finitely
  generated ideal is of finite definition. For instance, consider $\^P=k[[x_n
  \mid n\in\N]]$. The principal ideal generated by $f=\sum_{n\geq1} x_n^n$
  is of finite definition by \cref{p:inverse-fn-thm} but not in the
  indeterminates $x_i$. As for the second claim, an example is given by the
  principal ideal generated by $g=\sum_{n\geq1} x_n^{n+1}$, which, as we shall
  discuss next, is not of finite definition if $k$ is of
  characteristic $0$. Indeed, assume by contradiction that there exists an
  isomorphism $\^P \isom k[[y_n \mid n\in\N]]$ such that $g\^P$ is of finite
  definition with respect to the indeterminates $y_n$. Pick a variable $y_r$ not
  appearing in the generators for $g\^P$, and consider the regular continuous
  derivation $d = \partial/\partial y_r$ on $\^P$. Notice that $d(g) = 0$. By
  regularity, we have $d(x_m)\in \^P^\times$ for some $m\geq1$. Writing 
  $d(g)=\sum_{n\geq1} (n+1)x_n^n d(x_n)$, 
  we see that $\ord_{x_m}(d(g)) < \infty$, contradicting the fact that $d(g) =
  0$.
\end{remark}

Ideals of finite definition form a class of ideals of $\^P$ for which
\cref{q:in-fin-gen} has a positive answer. We only give a sketch of the proof
of the next lemma here and refer the reader to \cite[Section 1.5]{Chi20} for
details.

\begin{lemma}
  \label{l:in-for-fin-def}
  Let $\fa \subset \^P$ be an ideal such that there exists $J \subset I$ finite
  with $\fa = \fa_J \^P$. Then $\ini(\fa) = \ini(\fa_J) P$.
\end{lemma}

\begin{proof}
  Choose a local monomial order $\prec$ compatible with the standard filtration
  on $\^P$; for example, take the order defined by $x^a \prec x^b$ if $x^b
  <_{\grlex} x^a$, where $<_{\grlex}$ denotes the usual graded lexicographic
  order. Then $\prec$ restricts to a local monomial order $\prec_J$ on $\^P_J$.
  Choose a standard basis $S = \{f_1,\ldots,f_r\}$ of $\fa_J$ with respect to
  $\prec_J$. By \cite[Theorem 4.1]{Bec90} this is equivalent to $S$ being
  closed under $s$-series. Note that \cite[Theorem 4.1]{Bec90} extends directly
  to the case of infinitely many indeterminates and thus it follows immediately
  that $S$ is a standard basis of $\fa$ with respect to $\prec$. Clearly we
  have that $\ini(\fa) = (\ini(f_1),\ldots,\ini(f_r))$, which proves the claim.
\end{proof}

We are interested now in understanding heights of ideals of finite definition. Let us 
start by looking at their minimal primes.

\begin{proposition}
  \label{l:prime-extends-to-prime}
  If $\fa \subset \^P$ is an ideal of finite definition, then $\fa$ has a finite
  number of minimal primes, and each of them is of finite definition. More
  precisely, let $J \subset I$ be a finite subset and assume that $\fa = \fa_J
  \^P$. If $\fp \subset \^P$ is a minimal prime of $\fa$, then $\fp = \fp_J \^P$.
  Moreover, the assignment $\fp \mapsto \fp_J$ gives a bijection between the
  minimal primes of $\fa$ and the minimal primes of $\fa_J$.
\end{proposition}

\begin{proof}
  Notice that if $\fp \subset \^P$ is a prime ideal, then $\fp_J \subset \^P_J$
  remains prime. Moreover, by \cref{c:flatness-completion} we have that
  $\^P_J\to \^P$ is faithfully flat and thus $\fq = (\fq\^P)
  \cap \^P_J$ for any ideal $\fq \subset \^P_J$.
  It is therefore sufficient to show that, for every prime ideal
  $\fq \subset \^P_J$, the extension $\fq \^P$ is prime. By
  \cref{r:bourbaki-power-series} we have an injection $\^P \to (\^P_J)^{\N^{(I
  \setminus J)}}$. Since $J$ is finite, $\^P_J$ is Noetherian and $\fq$ is
  finitely generated. This implies that 
  $\fq \, (\^P_J)^{\N^{(I \setminus J)}}
  =
  \fq^{\N^{(I \setminus J)}}$,
  that is, the elements of the extension $\fq \, (\^P_J)^{\N^{(I \setminus J)}}$
  are precisely the Bourbaki power series that, when expanded in the variables indexed by
  $I \setminus J$, have coefficients in $\fq$. Therefore have an injection
  \[
    \^P/\fq \^P \inj (\^P_J/\fq)^{\N^{(I \setminus J)}}
  \]
  and the ring in the right hand side is clearly a domain. Thus $\fq \^P$ is
  prime.
\end{proof}

\begin{remark}
  In the setup of the proof of \cref{l:prime-extends-to-prime}, if $J$ is
  infinite then it is no longer true that $\fq \, (\^P_J)^{\N^{(I \setminus J)}} =
  \fq^{\N^{(I \setminus J)}}$ for an arbitrary prime $\fq \subset \^P_J$. For
  example, let $J = \N$, pick $i_0 \in I \setminus J$, let $\fq=\^\fm_J$ be the
  maximal ideal in $\^P_J$, and consider the series $f = \sum_{n \in \N} x_n
  x_{i_0}^n$. Then $f$ the belongs to $\fq^{\N^{(I \setminus J)}}$ but not to
  $\fq \, (\^P_J)^{\N^{(I \setminus J)}}$. We do not know if the extension $\fq
  \^P$ remains prime when $J$ is infinite.
\end{remark}

\begin{remark}
  \cref{l:prime-extends-to-prime} shows that the ideal $(x_i \mid i\in J)\^P$ is
  prime whenever $J$ is finite. Since colimits of prime ideals remain prime, one
  sees that $(x_i \mid i\in J)\^P$ is prime for an arbitrary subset $J$. In
  particular $\fm_0 = (x_i \mid i \in I)\^P$ is prime. Notice that $\^P/\fm_0$
  has infinite dimension when $I$ is infinite.
\end{remark}

The proof of the following \lcnamecref{r:ht(p_j)=ht(p)} uses 
the results of the previous section and \cref{c:aq-noetherian}.

\begin{theorem}
  \label{r:ht(p_j)=ht(p)}
  If $J \subset I$ is a finite subset, and $\fa = \fa_J \^P$, then $\height(\fa)
  = \height(\fa_J)$.
\end{theorem}

\begin{proof}
  From \cref{l:prime-extends-to-prime} we can assume that $\fa = \fp = \fp_J \^P$
  is a prime ideal. Notice that $\fp \subset \fb := (x_j \mid j\in J)$. From
  \cref{c:aq-noetherian} the localization $\^P_\fb$ is Noetherian, and therefore
  $\^P_\fp$, which is a further localization of $\^P_\fb$, is also Noetherian. By
  \cref{c:flatness-completion} the extension $\^P_J \subset \^P$ is flat, and
  therefore the extension $(\^P_J)_{\fp_J} \subset \^P_\fp$ is also flat. Since
  $\^P_\fp$ is Noetherian, it follows from by \cite[Theorem~15.1]{Mat89} that
  $\height(\fp) = \height(\fp_J)$.
\end{proof}

\begin{corollary}
  \label{p:finite-def-finite-ht}
  Let $\fa \subset \^P$ be any ideal of finite definition. For every minimal
  prime $\fp$ of $\fa$, we have $\height(\fp) < \infty$.
\end{corollary}

\begin{remark}
  \label{r:finite-def-properties-poly}
  In the case of polynomial rings the analogues of \cref{r:ht(p_j)=ht(p),p:finite-def-finite-ht} are 
  well known and easy to
  prove, and in fact there is a strong converse to the analogue of
  \cref{p:finite-def-finite-ht}, since
  every prime ideal of finite height in a polynomial ring $P = k[x_i \mid i \in
  I]$ is finitely generated. To see this, 
  suppose $\fp \subset k[x_i \mid i \in I]$ is a prime ideal that is not finitely
  generated. Recall that $\fp$ is the colimit of the ideals $\fp \cap P_J$ as $J$
  ranges among the finite subsets of $I$. This implies that we can fix an
  embedding $\N \subset I$ and find an increasing sequence $\{r_n \mid n \in \N\}
  \subset\N$ such that, if $\fp_n \subset k[x_i \mid i \in I]$ is the ideal
  generated by $\fp \cap k[x_1,\dots,x_{r_n}]$, then $\fp_n \subsetneq \fp_{n+1}$
  for all $n$. Since $\fp_n$ are all prime and are contained in $\fp$, it follows
  that $\height(\fp) = \infty$.

  Moreover, for arbitrary ideals $\fa$ of $P$ it is proven in \cite[Theorem
  3.3]{GH93} that $\fa$ is finitely generated if and only if it has finitely
  many associated primes, each of which is of finite height.
\end{remark}

\begin{proposition}
  \label{c:aq-noetherian}
  For every finite $J\subset I$, the localization $\^P_{(x_j \mid j\in J)}$ is
  Noetherian.
\end{proposition}

\begin{proof}
  As discussed in \cref{r:series-of-series}, since $J$ is finite we have an
  isomorphism
  \[
    \^P \isom k[[x_i \mid i\in I\setminus J]][[x_j \mid j\in J]].
  \]
  The \lcnamecref{c:aq-noetherian} now follows from the next \lcnamecref{p:aq-noetherian}.
\end{proof}

\begin{lemma}
  \label{p:aq-noetherian}
  For any $n\in\N$, let $\^P_n:=\^P[[y_1,\ldots,y_n]]$ and consider the ideal
  $\fb_n:=(y_1, \ldots, y_n)$ in $\^P_n$. Then the localization $(\^P_n)_{\fb_n}$
  is a Noetherian ring. 
\end{lemma}

The proof of \cref{p:aq-noetherian} uses the following straightforward
generalization of the Weierstrass division theorem, whose proof is a simple
adaptation of the proof of \cite[VII, §3.8]{Bou72} where the adic topology on
$\^P_n$ is replaced with the inverse limit topology. We say that $f\in \^P_n$
is \emph{$y_n$-regular} of order $d$ if its image under the canonical map
$\^P_n \to k[[y_n]]$ is nonzero of order $d$.

\begin{theorem}
  \label{p:weierstrass-division}
  Let $f\in \^P_{n+1}$ be $y_{n+1}$-regular of order $d$. For every $g\in
  \^P_{n+1}$ there exist unique $q\in \^P_{n+1}$ and $r\in \^P_{n}[y_{n+1}]$ such
  that $g=q f+r$ and $r$ has degree $< r$ as a polynomial in $y_{n+1}$.
\end{theorem}

The next lemma ensures that we can apply \cref{p:weierstrass-division} to prove
\cref{p:aq-noetherian}.

\begin{lemma}
  \label{l:coordinate-change-regular}
  Let $f\in \^P_{n+1}=\^P[[y_1,\ldots,y_{n+1}]]$ be a nonzero element. Then there
  exists a continuous $k$-automorphism $\vp \colon \^P_{n+1} \to \^P_{n+1}$ such
  that $\vp(\fb_{n+1})=\fb_{n+1}$ and $\vp(f)$ is $y_{n+1}$-regular.  
\end{lemma}

\begin{proof}
  If $f$ is already $y_{n+1}$-regular, then we are done. If not, then pick any
  monomial of the form
  $x_{i_1}^{d_1}\cdots x_{i_r}^{d_r} y_1^{e_1}\cdots y_{n+1}^{e_{n+1}}$
  appearing in the expansion of $f$. Then decompose $f$ as
  \[
    f=f'+f'', \quad f'\in k[[x_{i_1},\ldots,x_{i_r},y_1,\ldots,y_{n+1}]],
  \]
  such that $f'$ cannot be decomposed further as above. By \cite[Lemma
  6.11]{Art69} there exist new coordinates
  $x'_{i_j} = x_{i_j}+y_{n+1}^{a_j}$,
  $y'_l = y_l+y_{n+1}^{b_l}$, and $y'_{n+1} := y_{n+1}$.
  such that $f'(x'_{i_j},y'_l)$ is $y'_{n+1}$-regular. We may extend this change
  of coordinates trivially to a continuous automorphism $\vp \colon \^P_{n+1} \to
  \^P_{n+1}$ by setting $x'_i=x_i$ for all indices $i$ that are different from
  $i_j$ for all $j$. Then clearly $\vp(f)$ is $y_{n+1}$-regular and $\vp$ fixes
  $\fb_{n+1}=(y_1,\ldots,y_{n+1})$.
\end{proof}

\begin{proof}[Proof of \cref{p:aq-noetherian}]
  We prove the \lcnamecref{p:aq-noetherian} by induction on $n$. Let $Q_n:=(\^P_n)_{\fb_n}$.
  Clearly $Q_0\isom \Quot(\^P)$, so let us assume that $Q_n$ is Noetherian. We
  have injections $Q_n \to Q_{n+1}$. Let $\fa$ be an ideal of $Q_{n+1}$ and
  $f\in\fa$, $f\neq0$. After multiplication by a unit we may assume $f\in
  \^P_{n+1}$; by \cref{l:coordinate-change-regular} we may also assume $f$ is
  $y$-regular. Consider the ideal $\fa':=\fa\cap Q_n[y_{n+1}]$. Since $Q_n$ is
  Noetherian, so is $Q_n[y_{n+1}]$ and thus there exist $f_1,\ldots,f_r\in
  Q_n[y_{n+1}]$ that generate $\fa'$. We claim that
  $\fa=(f,f_1,\ldots,f_r)Q_{n+1}$.

  Let $g\in\fa$. By \cref{p:weierstrass-division}, there exist a unit $u\in
  Q_{n+1}$, $q\in \^P_{n+1}$, and $r\in \^P_n[y_{n+1}]$ such that $u g=q f+r$.
  Since $r\in\fa'$, we can find $v_1,\ldots,v_r\in Q_n$ such that
  $r=\sum_{j=1}^r v_j f_j$.
  Hence we have
  $g=u^{-1}q f+\sum_{j=1}^r u^{-1}v_j f_j$,
  which proves our claim.
\end{proof}


\section{Embedding codimension}

\label{s:embcodim}

Let $(A, \fm, k)$ be a local ring. The
inclusion of $\fm/\fm^2$ in the graded ring $\gr(A)$ induces a natural
surjective homomorphism of $k$-algebras
\[
  \g \colon \Sym_k(\fm/\fm^2) \to \gr(A).
\]

\begin{definition}
  \label{d:embcodim}
  The \emph{embedding codimension} of $(A, \fm, k)$ is defined to be
  \[
    \embcodim(A) := \height(\ker(\g)).
  \]
\end{definition}

\begin{proposition}
  \label{p:embdim=grcodim+dim}
  For any local ring $(A, \fm, k)$, we have
  \[
    \embdim(A) = \dim(\gr(A)) + \embcodim(A).
  \]
  In particular, if $A$ is Noetherian then 
  $\embdim(A) = \dim(A) + \embcodim(A)$.
\end{proposition}

\begin{proof}
  This follows from the fact that for every
  polynomial ring $P =  k[x_i \mid i \in I]$ and every ideal $\fa \subset P$, 
  we have $\dim(P) = \dim(P/\fa) + \height(\fa)$
  (cf.\ \cref{r:finite-def-properties-poly}).
  For the last assertion we use that $\dim(\gr(A)) = \dim(A)$ if $A$ is Noetherian.
\end{proof}

\begin{remark} 
The formula in \cref{p:embdim=grcodim+dim} is still valid, and informative, when some
of the quantities involved are infinite.
\end{remark}

\begin{remark}
  Higher rank valuation rings provide examples of finite dimensional 
  non-Noetherian rings whose embedding dimension 
  is smaller than their dimension. 
  For example, let $A \subset k(x,y)$ be the valuation ring associated to the valuation
  $v \colon k(x,y)^* \to \Z^2_{\rm lex}$ given by $v(f) = (\ord_x(f), \ord_y(fx^{-\ord_x(f)}|_{x=0}))$. 
  This is a two-dimensional ring whose maximal ideal is principal, which implies that the embedding
  dimension is one. 
  In particular, 
  the second equation in \cref{p:embdim=grcodim+dim}
  does not hold for such rings.
\end{remark}

\begin{remark}
%
The embedding codimension of a local ring was studied in the Noetherian setting
in \cite{Lec64} under the name of \emph{regularity defect}. One of the main results proved
there is that if $\fp$ is a prime ideal of a Noetherian local ring $(A,\fm)$
such that $\dim(A) = \dim(A/\fp) + \dim(A_\fp)$, then $\embcodim(A_\fp) \le
\embcodim(A)$, see \cite[Theorem~3]{Lec64}. It would be interesting to find
suitable conditions for the same property to hold in the non-Noetherian
setting.
\end{remark}

We now come to the main result of this section, which gives bounds for the
embedding codimension of $A$ from maps into $A$.

\begin{proposition}
  \label{p:grcodim-projection}
  Let $\f \colon (B,\fn,k_0) \to (A,\fm,k)$ be a homomorphism of local rings, and assume that
  $(B,\fn)$ has finite embedding dimension. Let $\f^* \colon \fn/\fn^2 \otimes_{k_0}k \to
  \fm/\fm^2$ be the induced $k$-linear map on the Zariski cotangent spaces. Then 
  \[
  \embcodim(A) \ge \rank(\f^*) - \dim(\gr(B)).
  \]
  In particular, if $B$ is Noetherian then
  $\embcodim(A) \ge \rank(\f^*) - \dim(B)$.
\end{proposition}

\begin{remark}
  \label{r:can-assume-phi-inj}
  A stronger form of \cref{p:grcodim-projection} 
  is obtained by replacing $\dim(\gr(B))$ with
  $\dim(\gr(B/\ker(\gr(\f))))$ in the displayed formula.
  Note, in fact, that this
  sharper form of the \lcnamecref{p:grcodim-projection} 
  follows from the special case of the \lcnamecref{p:grcodim-projection} where
  $\f$ is assumed to be injective.
\end{remark}

\begin{remark}
  Consider the special case where $\f$ is a homomorphism of local $k$-algebras
  with residue fields $k$ (that is, such that the natural maps $k \to B/\fn$ and
  $k \to A/\fm$ are isomorphisms) and with $B$ essentially of finite type. The
  geometric interpretation is the following. Let $f \colon X \to Y$ be a morphism
  of schemes over $k$, with $Y$ of finite type over $k$, and let $p \in X(k)$ and
  $q = f(p) \in Y(k)$. Denote by $T_pf \colon T_pX \to T_qY$ the map induced on
  Zariski tangent spaces. Then the \lcnamecref{p:grcodim-projection} gives
  \[
    \embcodim(\O_{X,p}) \ge \dim(\im(T_pf)) - \dim_q(\im(f))
  \]
  where $\im(f) \subset Y$ is the scheme-theoretic image of $f$. Note in
  particular that if $X$ is Noetherian then this formula reduces to the intuitive
  statement that
  \[
    \dim(T_pX) - \dim_p(X) \ge \dim(\im(T_pf)) - \dim_q(\im(f)).
  \]
  Another special case is when $f$ is a submersion onto $Y$, in which case the formula reduces to
  the inequality $\embcodim(\O_{X,p}) \ge \embcodim(\O_{Y,q})$.
\end{remark}

\begin{proof}[Proof of \cref{p:grcodim-projection}]
  We have the commutative diagram
  \[
    \xymatrix{
      \Sym_{k_0}(\fn/\fn^2) 
      \ar[r]^\p
      \ar@{->>}[d]
      & \Sym_k(\im(\f^*)) 
      \ar[d]_{\s}
      \ar@{^(->}[r]^\iota
      & \Sym_k(\fm/\fm^2)
      \ar@{->>}[d]_\g
      \\
      \gr(B)
      \ar[r]
      \ar@/_1.5pc/[rr]_{\gr(\f)}
      & \im(\gr(\f)) \otimes_{k_0}k
      \ar@{^(->}[r]^(.6){\ff}
      & \gr(A)
      \\
    }
  \]
  The existence of $\s$ follows from the fact that 
  $\im(\p) \otimes_{k_0}k = \Sym_k(\im(\f^*))$. 
  The map $\iota$ is a linear extension of polynomial rings, and hence is faithfully flat.
  Since $\iota^{-1}(\ker(\g)) = \ker(\s)$ we see that
  \[
    \height(\ker(\g)) \geq \height(\ker(\s))
  \]
  by the going-down theorem. On the other hand, 
  \[
    \height(\ker(\s))
    =
    \rank(\f^*) - \dim(\im(\s)).
  \]
  Since the inclusion $\im(\s) \subset \im(\gr(\f)) \otimes_{k_0}k$ 
  is an inclusion of Noetherian local rings with the
  same residue field, and $\im(\gr(\f))$ is a quotient of 
  $\gr(B)$, we have
  \[
    \dim(\im(\s)) \le \dim(\gr(B)).
  \]
  Combining the above formulas, we get
  \[
    \height(\ker(\g))
    \geq
    \rank(\f^*) - \dim(\gr(B)).
  \]
  To conclude, notice that $\dim(\gr(B)) = \dim(B)$ if $B$ is Noetherian.
\end{proof}

The following result shows that the embedding codimension of $A$ is invariant
under change of the base field, provided the residue field is already contained
in $A$.

\begin{proposition}
  \label{p:grcodim-field-change}
  Let $(A,\fm,k)$ be a local $k$-algebra 
  such that the natural map $k \to A/\fm$ is an isomorphism, 
  and let $k \subset k'$ be a field extension. 
  Denoting $A' := A \otimes_k k'$, we have
  \[
    \embcodim(A') = \embcodim(A).
  \]
\end{proposition}

\begin{proof}
  First, observe that $A'$ is a local $k'$-algebra with maximal ideal $\fm' = \fm \otimes_k k'$.
  We have $\embcodim(A) = \height(\ker(\g))$, where 
  \[
    \g \colon \Sym_k(\fm/\fm^2) \to \gr(A)
  \]
  is defined, as at the beginning.
  Since for every $n$ we have $(\fm')^n/(\fm')^{n+1} = \fm^n/\fn^{n+1} \otimes_k k'$, we see that $\g$
  induces, by base change, the analogous map
  \[
    \g' \colon \Sym_k(\fm'/(\fm')^2) \to \gr(A').
  \]
  The next \lcnamecref{l:same-ht} gives $\height(\ker(\g')) = \height(\ker(\g))$
  and the assertion follows.
\end{proof}

\begin{lemma}
  \label{l:same-ht}
  Let $P = k[x_i \mid i \in I]$ and $P' = P \otimes_kk' = k'[x_i \mid i \in I]$, 
  where $k \subset k'$ is a field extension.
  Then for every ideal $\fa \subset P$ we have $\height(\fa) = \height(\fa P')$.
\end{lemma}

\begin{proof}
  For short, let $\fa' = \fa P'$. 
  If $I$ is finite, then the \lcnamecref{l:same-ht} follows from dimension theory.
  In general, 
  suppose by contradiction that $\height(\fa) \ne \height(\fa')$. 
  Then we can find a finite subset $J \subset I$
  such that $\height(\fa_J) \ne \height(\fa'_J)$
  (cf.\ \cref{r:finite-def-properties-poly}). 
  Since $\fa'_J = \fa_J P'_J$, this contradicts the finite dimensional case. 
\end{proof}


\section{Formal embedding codimension}

\label{s:fembcodim}

In the case of equicharacteristic local rings, looking at the completion
instead of the associated graded provides a different way of defining 
embedding codimension. To distinguish the two, we introduce the following 
terminology.

\begin{definition}
  \label{d:embcodim-1}
  The \emph{formal embedding codimension} of an equicharacteristic local ring $(A, \fm,
  k)$ is defined to be
  \[
    \fembcodim(A) := \inf_{\t} \height(\ker(\t))
  \]
  where the infimum is taken over all choices of formal embeddings $\tau \colon
  \^P \to \^A$ (see \cref{def:formal-emb}).
\end{definition}

\begin{proposition}
  In the above \lcnamecref{d:embcodim-1}, 
  we have $\fembcodim(A) = \height(\ker(\t))$
  for every efficient formal embedding $\tau \colon \^P \to \^A$.
\end{proposition}

\begin{proof}
  Given two formal embeddings
  $\t \colon \^P \to \^A$ and $\t' \colon \^P' \to \^A$
  with $\t$ efficient, by \cref{r:emb-indep-basis}
  there is a surjection $\f \colon \^P' \to \^P$ such that $\t' = \t \o \f$, 
  and hence $\height(\ker(\t')) \ge \height(\ker(\t))$. 
\end{proof}

\begin{remark}
  \label{p:grcodim=0}
  If $A$ is a local $k$-algebra such that the residue field $A/\fm$
  is separable over $k$, then it follows by \cite[Chapter~$0_{\text{IV}}$, Corollary~19.5.4]{EGAiv_1}
  that the following are equivalent:
  \begin{enumerate}
    \item
    $A$ is formally smooth over $k$.
    \item
    $\embcodim(A) = 0$.
    \item
    $\fembcodim(A) = 0$.
  \end{enumerate}
\end{remark}

\begin{proposition}
  \label{p:embdim=embcodim+dim}
  For every equicharacteristic local ring $(A,\fm,k)$, we have
  \[
    \embdim(A) \ge \dim(\^A) + \fembcodim(A),
  \]
  and equality holds if $A$ has finite embedding dimension.
  In particular, if $A$ is Noetherian then 
  $\embdim(A) = \dim(A) + \fembcodim(A)$.
\end{proposition}

\begin{proof}
  Consider an efficient formal embedding $\t \colon \^P = k[[x_i \mid i\in I]] \to \^A$. 
  Note that $\dim(\^P) = \embdim(A)$ by \cref{th:embdim=inf}. 
  The first formula follows from the simple
  fact that $\dim(\^P) \ge \height(\ker(\t)) + \dim(\^P/\ker(\t))$.
  If $A$ has finite embedding dimension then the set $I$ is finite, and 
  equality holds in the formula because a power series ring in finitely many
  variables is catenary of dimension equal to the number of variables. The second
  formula follows from the first and the fact that $\dim(A) = \dim(\^A)$ if $A$
  is Noetherian.
\end{proof}

\begin{corollary}
  \label{p:grcodim-embcodim:finite-embdim}
  If $(A,\fm, k)$ is an equicharacteristic local ring of finite embedding
  dimension, then 
  \[
    \embcodim(A) = \fembcodim(A).
  \]
\end{corollary}

\begin{proof}
  By \cref{p:embdim=embcodim+dim,p:embdim=grcodim+dim}, it suffices to show that
  $\dim(\gr(A)) = \dim(\^A)$. By \cref{l:tau-exists} the completion $\^A$ is the
  quotient of a power series ring in finitely many variables, and therefore
  is Noetherian and carries the $\^\fm$-adic topology. The result now follows from
  \cite[Theorem~15.7]{Mat89} and the identification $\gr(\^A) \isom \gr(A)$.
\end{proof}

\begin{proposition}
  \label{p:DGK-grcodim}
  Let $(A, \fm, k)$ be an equicharacteristic local ring. 
  If $A$ admits a DGK decomposition $\^A \isom \^B \,\cotimes_k\, \^P$, then
  \[
    \embcodim(A) = \fembcodim(A) = \embcodim(B) < \infty.
  \]
\end{proposition}

\begin{proof}
  Since $B$ is Noetherian, we have $\embdim(B) < \infty$.
  By \cref{p:grcodim-embcodim:finite-embdim}, we have
  that $\embcodim(B) = \fembcodim(B)$. Now we make use of the fact that $A
  \isom k[[x_i \mid i \in I]]/\fa$ with $\fa$ of finite definition. Note that
  there exists $J \subset I$ finite such that $\fa$ is the extension of $\fa_J
  := \fa \cap k[[x_j \mid j \in J]]$ and $\^B \isom k[[x_j \mid j \in
  J]]/\fa_J$. We may assume that the surjection $\tau_{B} \colon k[[x_j \mid j
  \in J]] \to \^B$ is an efficient formal embedding; then so is $\tau_{A}
  \colon k[[x_i \mid i \in I]] \to \^A$. By
  \cref{r:DGK-and-finite-def,r:ht(p_j)=ht(p)} it follows that $\fembcodim(B) =
  \fembcodim(A)$. It remains to show that $\embcodim(B) = \embcodim(A)$.
  
  To that avail, note that $\gr(\tau_A)$ factors through the natural surjection
  $\Sym_k(\fm/\fm^2) \to \gr(A)$, and similarly for $\gr(\tau_B)$. We have the
  commutative diagram
  \[
    \xymatrix{k[x_i \mid i \in I] \ar[r]^-{\isom} & \Sym_k(\fm/\fm^2) \ar[r] & \gr(A) \\
              k[x_j \mid j \in J] \ar[r]^-{\isom} \ar[u] & \Sym_k(\fn/\fn^2) \ar[r] \ar[u] & \gr(B). \ar[u]
    }
  \]
  and thus the claim follows by \cref{l:in-for-fin-def}.
\end{proof}

\begin{remark}
  The analogous statement of \cref{p:DGK-grcodim} holds for equicharacteristic
  local rings $(A,\fm,k)$ admitting a weak DGK decomposition.	
\end{remark}

The proof of \cref{p:grcodim-embcodim:finite-embdim} 
does not extend beyond the case of finite embedding dimension. 
Nonetheless, the following general comparison \lcnamecref{t:grcodim-embcodim} holds.

\begin{theorem}
  \label{t:grcodim-embcodim}
  For every equicharacteristic local ring $(A,\fm,k)$, we have
  \[
    \embcodim(A)
    \leq
    \fembcodim(A).
  \]
\end{theorem}

\begin{proof}
  Fix an efficient formal embedding $\t \colon \^P \to \^A$,
  and let $\gr(\t) \colon P \to \gr(A)$ be the induced map
  on associated graded rings (as in \cref{r:gr-are-poly} we identify $\gr(\^P) =
  \gr(P) = P$). As explained in \cref{r:tau-using-Sym}, $P \isom
  \Sym_k(\fm/\fm^2)$ and $\gr(\t)$ gets identified with the canonical surjection
  $\g$. In particular, it is enough to show that
  \[
    \height(\ker(\t)) \geq \height(\ker(\gr(\t))).
  \]
  Write $\fa = \ker(\t)$.
  By \cite[Chapter~III, Section~2.4, Proposition~2]{Bou72}, we have that
  $\ker(\gr(\t)) = \ini(\fa)$. To conclude, it is therefore enough to prove that
  \[
    \height(\fa) \ge \height(\ini(\fa)).
  \]
  This follows from the next \lcnamecref{t:degeneration-to-gr}.
\end{proof}

\begin{proposition}
  \label{t:degeneration-to-gr}
  Let $P = k[x_i \mid i \in I]$ and $\^P = k[[x_i \mid i \in I]]$, where $k$ is a
  field. Let $\fa \subset \^P$ be an ideal and $\ini(\fa) \subset P$ the
  corresponding initial ideal. Then 
  $\height(\fa) \ge \height(\ini(\fa))$.
\end{proposition}

\begin{proof}
  If $\height(\ini(\fa)) < \infty$ then $\ini(\fa)$ is finitely generated and
  there exists a finite subset $J \subset I$ such that $\ini(\fa)_J$ has the
  same height of $\ini(\fa)$, see \cref{r:finite-def-properties-poly}.
  Otherwise, if $\height(\ini(\fa)) = \infty$ then we can pick a finite $J
  \subset I$ such that the height of $\ini(\fa)_J$ is arbitrary large. Let $c
  := \height(\ini(\fa)_J)$.

  By \cite[Proposition~1.5.11]{BH93}, we can fix homogeneous elements $g_1,\dots,g_c \in \ini(\fa)_J$ 
  forming a regular sequence in $P_J$. 
  By the definition of initial ideal, there are elements
  $f_1,\dots,f_c \in \fa$ such that $\ini(f_i) = g_i$ for all $i$. 
  Let $R = R(\^P) := \bigoplus_{n \in \Z} \^{\fm^n}u^{-n}$ be the extended Rees
  algebra of $\^P$, where we set $\^{\fm^n} = \^P$ whenever $n < 0$. 
  For every $i$, let $\~f_i := u^{-\ord(f_i)}f_i \in R$.
  Note that $\~f_i|_{u=0} = \ini(f_i) = g_i$ via the identification $R/uR \isom P$.

  We claim that, for every $1 \le r \le c$, the elements $\~f_1,\dots,\~f_r$
  form a regular sequence in $R$ and $R/(\~f_1,\dots,\~f_r)$ is flat over $k[u]$. 
  We argue by induction on $r$, the assertion being clear if $r=0$. 
  Letting for short $B := R/(\~f_1,\dots,\~f_{r-1})$, 
  we know by induction that $B$ is flat over $k[u]$. Assume that there exists $h\in\^P$ with $h=\sum_{i=1}^{r-1} a_if_i$
  and $\ini(h)$ not divisible by $g_1,\ldots,g_{r-1}$. Writing
  \[
    \sum_{i=1}^{r-1} u^{\ord(f_i)}a_i \~f_i=u^{\ord(h)}\~h,
  \]
  \cref{l:initial-ideal-rees-algebra} yields that $B$ has torsion over $k[u]$, which gives a contradiction.
  Thus $\ini(f_1,\ldots,f_{r-1})=(g_1,\ldots,g_{r-1})$ and $B$ 
  is isomorphic to the algebra
  $\bigoplus_{n \in \Z} \fb_n u^{-n}$ where
  $\fb_n := (\^{\fm^n} + (f_1,\dots,f_{r-1}))/(f_1,\dots,f_{r-1})$ for $n \ge 0$
  and $\fb_n = B$ for $n < 0$. 
  It follows by \cref{p:f.g.initial} that 
  $\bigcap_{n \ge 1} \fb_n = \{0\}$.
  Then \cref{th:u-adically-separated} implies that $B$ is $(u)$-adically separated, 
  and \cref{th:lifting-regularity-flatness} (with $t = u$)
  implies that the class $b$ of $\~f_r$ in $B$ is a regular element and $B/bB$
  is flat over $k[u]$.

  The natural isomorphism $R/(u-1)R \isom \^P$ sends $\~f_i$ to $f_i$, and hence
  we see by \cref{th:lifting-regularity-flatness} (with $t = u-1$)
  that $f_1,\dots,f_c$ form a regular sequence in $\^P$.
  This implies that $\depth(\fa,\^P) \ge c$. We conclude using the fact that
  $\height(\fa) \ge \depth(\fa,\^P)$
  (e.g., see \cite[Proposition~2.3 and Lemma~3.2]{AT09}).
\end{proof}

\begin{lemma}
  \label{l:initial-ideal-rees-algebra}
  Let $f_1,\ldots,f_r\in \^P$. If $h\in \^P$, then
  $\ini(h)\in(\ini(f_1),\ldots,\ini(f_r))$ if and only if there exist elements
  $b_1,\ldots,b_r$ in the Rees algebra $R=R(\^P)$ such that
  \[
    \~h=\sum_i^r b_i \~f_i.
  \]
\end{lemma}

\begin{proof}
  Given $\~h=\sum b_i \~f_i$, we may assume that $b_i$ is homogeneous in $R$,
  i.e.\ of the form $b_i=u^{-\ord(h)+\ord(f_i)}a_i$ with $a_i\in \^P$. But then
  $\ord(a_i)\geq\ord(h)-\ord(f_i)$ and the claim follows.
\end{proof}

\begin{lemma}
  \label{th:u-adically-separated}
  Let $A$ be a ring and $(\fa_n)_{n \ge 0}$ a graded sequence of ideals of $A$, and
  let $R(A) := \bigoplus_{n \in \Z} \fa_n u^{-n}$ where we set $\fa_n = A$ for $n < 0$. 
  Assume that $\bigcap_{n \ge 1} \fa_n = \{0\}$. 
  Then $R(A)$ is $(u)$-adically separated. 
\end{lemma}

\begin{proof}
  Let $a \in R(A)$ be any element.
  Write $a = \sum_{i=p}^q a_i u^{-i}$ for some $a_i \in \^P$ and $p,q \in \Z$. 
  By the definition of Rees algebra, we have $a_i \in \fa_i$ for all $i$.
  The condition that $a \in u^nR(A)$ is equivalent to having $a_i \in \fa_{n+i}$ for all $i$.
  If $a \in \bigcap_{n \ge 1} u^nR(A)$, then we have
  $a_i \in \bigcap_{n \ge 1}\fa_n = \{0\}$ for all $i$, and hence $a = 0$. 
\end{proof}

\begin{lemma}
  \label{th:lifting-regularity-flatness}
  Let $B$ be a flat and  $k[t]$-algebra. 
  For any given $b \in B$, consider the following properties:
  \begin{enumerate}
    \item
    \label{item1:lifting-regularity-flatness}
    $b$ is a regular element of $B$ and $B/bB$ is flat over $k[t]$.
    \item
    \label{item2:lifting-regularity-flatness}
    The image $\bar{b}$ of $b$ in $B/tB$ is regular. 
  \end{enumerate}
  Then $\eqref{item1:lifting-regularity-flatness} \Rightarrow \eqref{item2:lifting-regularity-flatness}$, and the converse holds if $B$ is
  $(t)$-adically separated.
\end{lemma}

\begin{proof}
  The proof is an adaptation of the proof of \cite[Theorem~22.5]{Mat89}. 
  The implication $\eqref{item1:lifting-regularity-flatness} \Rightarrow
  \eqref{item2:lifting-regularity-flatness}$ follows by the snake lemma applied
  to the commutative diagram\[
  \xymatrix{
  B \ar[r]^{b} \ar[d]^{t} & B \ar[d]^{t} \ar[r] & B/bB \ar[d]^{t}\\\
  B \ar[r]^{b} \ar[d] & B \ar[r] \ar[d] & B/bB  \\
  B/tB \ar[r]^{\bar b} & B/tB & 
  }
  \]
  after observing that the map $B \to B$ given by multiplication by $b$ is injective
  since $b$ is regular, and so is the map $B/bB \to B/bB$ given by multiplication by 
  $t$ since $B/bB$ is flat over $k[t]$. 

  In order to prove the implication $\eqref{item2:lifting-regularity-flatness}
  \Rightarrow \eqref{item1:lifting-regularity-flatness}$ when $B$ is
  $(t)$-adically separated, 
  suppose $x \in B$ is an element such that $bx = 0$. Then $\bar{b}\bar{x} = 0$ in $B/tB$
  and hence $\bar{x} = 0$. This means that $x \in tB$. 
  Suppose $x \in t^nB$ for some positive integer $n$, and write $x = t^ny$ in $B$. 
  Then $t^n(by) = bx = 0$ and hence $by = 0$ since $B$ is flat over $k[t]$. 
  This implies that $y \in tB$, and hence $x \in t^{n+1}B$. 
  Therefore $x \in \bigcap_{n \ge 1} t^nB$, and since $B$ is $(t)$-adically separated, 
  this means that $x = 0$. This proves that $b$ is a regular element. 
  To conclude that $B/bB$ is flat over $k[t]$, we just compute that
  $\Tor^{k[t]}_1(k,B/bB) = 0$ from the exact sequence 
  $0 \to B \to B \to B/bB \to 0$ and apply \cite[Theorem 22.3]{Mat89}.
\end{proof}

\begin{question}
  We do not know of any example where the inequality in \cref{t:grcodim-embcodim} is strict. 
  The question whether 
  $\embcodim(A)=\fembcodim(A)$ holds for all equicharacteristic local rings
  $(A,\fm,k)$ is, to our knowledge, still open.
\end{question}


\section{Embedding codimension of arc spaces}

\label{s:genproj}

Let $X$ be a scheme of finite type over a field $k$. The \emph{arc space} $X_\infty$
of $X$ is the scheme over $k$ representing the functor of points given, for any
$k$-algebra $R$, by $R \mapsto \liminv_m X_m(R)$, where
$X_m(R)=\Hom_k(\Spec(R[t]/(t^{m+1})),X)$ is the functor of points of the $m$-th
jet scheme of $X$. By \cite[Remark~4.6]{Bh16}, the functor $X_\infty(R)$ is
naturally isomorphic to $\Hom_k(\Spec(R[[t]]),X)$.
A point $\a \in X_\infty$ is called an \emph{arc on $X$} and corresponds to a
morphism $\Spec (L[[t]]) \to X$ where $L$ is the residue field of $\a$. 
A point $\a \in X_\infty$ is said to be \emph{constructible}
if $\a$ is the generic point of an irreducible constructible subset of $X_\infty$
(cf.\ \cite[Section~10]{dFD}).

Given an arc $\a \colon \Spec (k[[t]]) \to X$, we will denote by $\a(0)$ and $\a(\e)$ the
images in $X$ of the closed point and the generic point of $\Spec (k[[t]])$; 
we call $\a(0)$ the \emph{special point} of $\a$ and $\a(\e)$ the \emph{the generic point} of $\a$. 

Given an open set $U \subset X$, we have $\a(\e) \in U$
if and only if the morphism $\a \colon \Spec (k[[t]]) \to X$ does not
factor through the complement $X \setminus U$. 
We will be interested in the case where $U = X_\sm$, the smooth locus of $X$. 
Note that if $k$ is perfect, then the complement $X \setminus X_\sm$
is the singular locus $\Sing(X)$ of $X$.

We following result is a variant of \cite[Theorems~9.2 and~9.3]{dFD}.

\begin{theorem}
  \label{t:general-proj}
  Suppose that $X$ is an affine scheme over a perfect field $k$. 
  Let $\a \in X_\infty$ be an arc and let $d:=\dim_{k(\a(\e))}(\Omega_{X/k}\otimes_k k(\a(\e)))$
  where $k(\a(\e))$ is the residue field of $\a(\e) \in X$.
  Assume that one of the following occurs:
  \begin{enumerate}
    \item
    \label{item1:general-proj}
      $k$ is a field of characteristic zero, or
    \item
    \label{item2:general-proj}
      $\a \in X_\infty(k)$. 
  \end{enumerate}
  Fix a closed embedding $X \subset \A^N$, let 
  $f \colon X \to Y :=\A^d$ 
  be the morphism induced by a general linear projection $\A^N \to \A^d$, 
  and let $\b := f_\infty(\a) \in Y_\infty$. Let
  $\fm \subset \O_{X_\infty,\a}$ and $\fn \subset \O_{Y_\infty,\b}$
  be the respective maximal ideals and $L$ and $L'$
  the residue fields.
  Then the induced $L$-linear map
  \[
    (T_\a f_\infty)^* \colon \fn/\fn^2 \otimes_{L'}L \to \fm/\fm^2
  \]
  is an isomorphism. 
\end{theorem}

\begin{proof}
  By assumption, we have that $\ord_\a(\Fitt^d(\Omega_{X/k}))<\infty$ and by
  taking a general linear projection we can ensure that
  $\ord_\a(\Fitt^d(\Omega_{X/k}))=\ord_\a(\Fitt^0(\Omega_{X/Y}))$.

  Since $k$ is perfect, we have a commutative diagram with exact rows
  \[
      \xymatrix{
	  0 \ar[r]
	  & \fn/\fn^2 \otimes_{L'}L \ar[r] \ar[d]^{(T_\a f_\infty)^*}
	  & \Omega_{Y_\infty/k} \otimes_{\mathcal O_{Y_\infty}} L \ar[r] \ar[d]^{\f}
	  & \Omega_{L'/k} \otimes_{L'}L \ar[r] \ar[d]^\d
	  & 0 \\
	  0 \ar[r]
	  & \fm/\fm^2 \ar[r]
	  & \Omega_{X_\infty/k} \otimes_{\mathcal O_{X_\infty}} L \ar[r]
	  & \Omega_{L/k} \ar[r]
	  & 0 \\
      }.
  \]

  The main step is to understand the map $\f$. 
  As in the proof of \cite[Theorems~9.2]{dFD}, denote for short
  $B_L := L[[t]]$ and $P_L := L(\:\!\!(t)\:\!\!)/tL[[t]]$.

  Note that, by \cite[Theorem~5.3]{dFD}, there are natural isomorphisms
  \[
    \Om_{X_\infty/k} \otimes_{\O_{X_\infty}} L \isom \Om_{X/k} \otimes_{\O_X} P_L
  \]
  and
  \[
    \Om_{Y_\infty/k} \otimes_{\O_{Y_\infty}} L \isom \Om_{Y/k} \otimes_{\O_Y} P_L.
  \]
  We will use these isomorphisms to study $\f$. 

  By pulling back the terms of the exact sequence
  \[
    \Om_{Y/k} \otimes_{O_Y} \O_X \to \Om_{X/k} \to \Om_{X/Y} \to 0
  \]
  along $\a$, we obtain the exact sequence
  \[
    \Om_{Y/k} \otimes_{\O_Y} B_L \to \Om_{X/k} \otimes_{\O_X} B_L 
    \to \Om_{X/Y} \otimes_{\O_X} B_L \to 0.
  \]
  Since $Y$ is smooth, we see that the term $F_Y := \Om_{Y/k} \otimes_{\O_Y} B_L$ is a free
  $B_L$-module. Write $\Om_{X/k} \otimes_{\O_X} B_L = F_X \oplus T_X$ where
  $F_X$ is free and $T_X$ is torsion. 
  Since $\ord_\a(\Fitt^0(\Omega_{X/Y})) < \infty$, the term $T_{X/Y} := \Om_{X/Y} \otimes_{\O_X} B_L$
  is a torsion $B_L$-module, and we get an exact sequence:
  \[
    0 \to F_Y \to F_X \oplus T_X \to T_{X/Y} \to 0.
  \]
  Since $P_L$ is a divisible $B_L$-module, tensoring with $P_L$ kills torsion, and
  hence the above sequence gives the exact sequence
  \[
    0 \to \Tor_1^{B_L}(T_X,P_L) \to \Tor_1^{B_L}(T_{X/Y},P_L) \to 
    F_Y \otimes_{B_L}P_L \xrightarrow{\,\f'\,} F_X \otimes_{B_L}P_L \to 0.
  \]
  Note that $\f' = \f$ under the aforementioned isomorphisms. 
  We have $\Tor_1^{B_L}(T_X,P_L) \isom T_X$, and this has dimension $\ord_\a(\Fitt^d(\Omega_{X/k}))$ over $L$.
  Similarly, $\Tor_1^{B_L}(T_{X/Y},P_L) \isom T_{X/Y}$ has dimension $\ord_\a(\Fitt^0(\Omega_{X/Y}))$ over $L$.
  Since these two dimensions are equal, the map $\f'$ in the sequence above is an isomorphism. 
  We conclude that $\f$ is an isomorphism.

  The surjectivity of $\f$ implies that $\d$ is surjective, and the injectivity of $\d$ follows from our  
  assumption that either \eqref{item1:general-proj} or \eqref{item2:general-proj} holds. We conclude that $(T_\a f_\infty)^*$ is an isomorphism.
\end{proof}

\begin{corollary}
  \label{c:general-proj-jets}
  Keeping the assumptions and notation from \cref{t:general-proj}, let $\a_n$,
  $\b_n$ denote the images of $\a$, $\b$ under the projections $X_\infty\to
  X_n$ and $Y_\infty\to Y_n$, and $\fm_n\subset \O_{X_n,\a_n}$, $\fn_n\subset
  \O_{Y_n,\b_n}$ denote the corresponding ideals with residue fields $L_n$,
  $L'_n$. Then the induced $L$-linear map
  \[
    (T_{\a_n} f_n)^* \colon \fn_n/\fn^2_n \otimes_{L'_n}L \to \fm_n/\fm_n^2 \otimes_{L_n} L
  \]
  is injective for all $n\in\N$.
\end{corollary}

\begin{proof}
 This follows from the diagram
 \[
  \xymatrix{\fn/\fn^2 \otimes_{L'} L \ar[r] &  \fm/\fm^2 \\
	    \fn_n/\fn_n^2 \otimes_{L'_n} L \ar[r] \ar[u] & \fm_n/\fm_n^2 \otimes_{L_n} L \ar[u]}
 \]
 and the fact that the top horizontal and left vertical arrows are injections.
\end{proof}

\begin{theorem}
  \label{t:grcodim-of-arcs}
  Let $X$ be a scheme of finite type over a perfect field $k$ and $\a \in X_\infty$.
  Assume that one of the following occurs:
  \begin{enumerate}
    \item
    $k$ is a field of characteristic zero, or
    \item
    $\a \in X_\infty(k)$. 
  \end{enumerate}
  Then we have
  \[
    \embcodim(\O_{X_\infty,\a}) \le 
    \limsup_{n \to \infty} \embcodim\big(\O_{X_n,\a_n}\big)
  \]
  where $\a_n$ is the image of $\a$ under the truncation map $\p_n \colon X_\infty \to X_n$.
\end{theorem}

\begin{proof}
  We can assume without loss of generality that $X$ is affine.
  Given a map 
  \[
    f \colon X \to Y := \A^d,
  \] 
  we let $\b := f_\infty(\a) \in Y_\infty$. 
  For every $n$, we denote by $\a_n \in X_n$ and $\b_n \in Y_n$ the images of $\a$
  and $\b$ at the respective $n$-jet schemes. 
  It is convenient, within this proof, to change notation from before and let 
  $A_\infty := \O_{X_\infty,\a}$
  and
  $B_\infty := \O_{Y_\infty,\b}$, 
  and denote by $\fm_\infty \subset A_\infty$ and $\fn_\infty \subset B_\infty$ the respective
  maximal ideals and by $L_\infty := A_\infty/\fm_\infty$ 
  and $L_\infty' := B_\infty/\fn_\infty$ the residue fields.
  Similarly, for every $n \in \N$, we let
  $A_n := \O_{X_n,\a_n}$
  and
  $B_n := \O_{Y_n,\b_n}$, and denote by $\fm_n \subset A_n$ and $\fn_n \subset B_n$ the respective
  maximal ideals and by $L_n := A_n/\fm_n$ 
  and $L_n' := B_n/\fn_n$ the residue fields. 

  Note that we have direct systems $\{A_n \to A_{n+1} \mid n \in \N\}$ and 
  $\{B_n \subset B_{n+1} \mid n \in \N \}$, and 
  $A_\infty = \dirlim_n A_n$ and $B_\infty = \dirlim_n B_n$.
  Moreover, we have commutative diagrams
  \[
  \xymatrix@C=20pt@R=20pt{
  B_\infty \ar[r]^{\f_\infty} & A_\infty \\
  B_n \ar[r]^{\f_n} \ar@{^(->}[u] & A_n \ar[u]
  }
  \]
  where $\fn_n = \f_n^{-1}(\fm_n)$, $\fn_n = \fn_\infty \cap B_n$, and $\fm_n = \fm_\infty \cap A_n$.
  
  For every $n \in \N \cup \{\infty\}$, let
  \[
  d\f_n \colon \fn_n/\fn_n^2 \otimes_{L'_n}L_n \to \fm_n/\fm_n^2
  \]
  be the induced $L_n$-linear map. 

  We pick $f$ as in \cref{t:general-proj}.
  For every $n \in \N \cup \{\infty\}$, there is an associated map of graded rings
  $\gr(\f_n) \colon \gr(B_n) \to \gr(A_n)$.
  We denote by 
  \[
    \ff_n \colon \gr(B_n) \otimes_{L'_n}L_\infty \to 
    \gr(A_n) \otimes_{L_n}L_\infty.
  \]
  the map induced by $\gr(\f_n)$ by the indicated base changes.

  Note that $\fm_\infty = \dirlim_n \fm_n$ and  
  hence $\fm_\infty^r = \dirlim_n \fm_n^r$ for all $r$. 
  Indeed, if $a \in \fm_\infty^r$ for some $r \ge 2$, then we can write
  $a = a_1\dots a_r$ with $a_i \in \fm_\infty$; then we can pick $n$ such that
  such that $a_i \in \fm_n$ for all $i$ and hence $a \in \fm_n^r$. 
  It follows that 
  \[
    \gr(A_\infty) = \dirlim_n \gr(A_n) \otimes_{L_n}L_\infty, 
  \]
  and similarly we have 
  \[
    \gr(B_\infty) = \dirlim_n \gr(B_n) \otimes_{L'_n}L'_\infty.
  \]
  Since $\fn_n^r = \f_n^{-1}(\fm_n^r)$ for all $r$, for every $n$ we have a commutative diagram
  \[
    \xymatrix@C=35pt@R=20pt{
      \gr(B_\infty) \otimes_{L'_\infty}L_\infty \ar[r]^(.6){\ff_\infty} 
      & \gr(A_\infty) \\
      \gr(B_n) \otimes_{L'_n}L_\infty \ar[r]^{\ff_n} \ar@{^(->}[u] 
      & \gr(A_n) \otimes_{L_n}L_\infty \ar[u]
    }.
  \]

  For short, let 
  $R_n := \gr(A_n) \otimes_{L_n}L_\infty$, 
  $S_n := \gr(B_n) \otimes_{L'_n}L_\infty$,
  and
  $K_n := \ker(\ff_n)$.

\begin{lemma}
  \label{l:ht-K=limsup}
  $\height(K_\infty) = \limsup_n \height(K_n)$.
\end{lemma}

\begin{proof}
  First, note that 
  $K_\infty = \dirlim_n K_n$.
  Indeed, the inclusion $K_\infty \supset \dirlim_n K_n$ is clear, and conversely, 
  if $b \in K_\infty$ and we fix $n \in \N$ such that 
  $b \in S_n$, then $\ff_m(b)$ is in the kernel of 
  $R_m \to R_\infty$ for all $m \ge n$
  and hence, since each $\ff_m(b)$ maps to $\ff_{m+1}(b)$ via 
  $R_m \to R_{m+1}$, it follows that $\ff_m(b)$
  is zero for $m \gg n$, which means that $b \in K_m$ for $m \gg n$. 

  We are now ready to prove that
  \[
    \height(K_\infty) = \limsup_n\height(K_n).
  \]
  Note that the maps
  $S_n \to S_\infty$ are extensions of polynomial rings over the same field $L_\infty$.
  Thus they are faithfully flat and hence for every prime
  $\fp_n \subset S_n$ its extension $\fp_n S_\infty$ is prime.

  Consider first the case where $\height(K_\infty) < \infty$ and let $\fp
  \subset S_\infty$ be a minimal prime over $K_\infty$ with
  $\height(\fp)=\height(K_\infty)$. By \cref{r:finite-def-properties-poly} we
  have that $\fp$ is finitely generated by elements $f_1,\ldots,f_r\in
  S_\infty$. For each $n>0$ let $\fp_n$ be any minimal prime over $K_n$
  contained in $\fp\cap S_n$. Then $\fp':=\dirlim_n \fp_n$ is a prime of
  $S_\infty$ with $K_\infty \subset \fp' \subset \fp$, so $\fp'=\fp$. Let
  $n_1>0$ be such that $f_1,\ldots,f_r\in \fp_{n_1}$, then $\fp_n=\fp\cap S_n$
  for $n\geq n_1$. Given any chain of primes
  \[
    (0) = \fq_0 \subsetneq \fq_1 \subsetneq \dots \subsetneq \fq_t = \fp \subset S_\infty,
  \]
  pick $s_i \in \fq_i \setminus \fq_{i-1}$, and fix $n_2$ such that
  $s_1,\dots,s_t \in S_{n_2}$. Then for every $n \ge \max\{n_1,n_2\}$
  we get a chain of primes
  \[
    (0) = \fq_0 \cap S_n \subsetneq \fq_1 \cap S_n 
    \subsetneq \dots \subsetneq \fq_t \cap S_n = \fp_n.
  \]
  Thus $\height(K_\infty) \le \limsup_n \height(K_n)$. 
  The other inequality follows by the going-down theorem applied to $S_n \to S_\infty$.

  If $\height(K_\infty) = \infty$, then, since $\height(K_\infty) \ge \height(K_nS_\infty)$, 
  a similar argument shows that the sequence $\{\height(K_n)\}_n$ is unbounded.
\end{proof}

We can now finish the proof of the theorem.
Since $B_n$ is formally smooth, for every $n \in \N \cup \{\infty\}$ 
the natural map
\[
  \Sym_{L_n'}(\fn_n/\fn_n^2) \to \gr(B_n)
\]
is an isomorphism (see \cref{p:grcodim=0} for the case $n=\infty$).
Furthermore, the diagram
\[
  \xymatrix{
    S_n \ar[r]^{\ff_n} & R_n \\
    \Sym_{L'_n}(\fn_n/\fn_n^2) \otimes_{L'_n}L_\infty \ar[r]^{\s_n} \ar[u]^\isom 
    & \Sym_{L_n}(\fm_n/\fm_n^2) \otimes_{L_n}L_\infty \ar[u]_{\g_n}
  }
\]
is commutative.

By \cref{t:general-proj}, the map $\s_\infty$ is an isomorphism, and hence
\[
  \height(K_\infty) = \height(\ker(\g_\infty)) = \embcodim(A_\infty).
\]
Similarly, by \cref{c:general-proj-jets} $\s_n$ is an injective $L_\infty$-linear map of
polynomial rings, and we have
\[
  \height(K_n) \le \height(\ker(\g_n)) = \embcodim(A_n).
\]
Then we conclude by \cref{l:ht-K=limsup}.
\end{proof}

\begin{theorem}
  \label{t:arc-finite-grcodim}
  Let $X$ be a scheme of finite type over a perfect field $k$ and $\a \in X_\infty$.
  Assume that either $k$ is a field of characteristic zero, or $\a$ is a $k$-rational point. 
  Then we have
  \[
    \embcodim(\O_{X_\infty,\a}) \le \ord_\a(\Fitt^d(\Om_{X/k}))
  \]
  where $d = \dim_{\a(\e)}(X)$. 
  In particular:
  \begin{enumerate}
    \item
  \label{item1:arc-finite-grcodim}
    If $X$ is a variety then $\embcodim(\O_{X_\infty,\a}) \le \ord_\a(\Jac_X)$.
    \item
  \label{item2:arc-finite-grcodim}
    If $\a(\e) \in X_\sm$ and $X^0 \subset X$ is the irreducible component containing $\a(\e)$, then 
    $\embcodim(\O_{X_\infty,\a}) \le \ord_\a(\Jac_{X^0}) < \infty$.
  \end{enumerate}
\end{theorem}

\begin{proof}
  First note that it suffices to prove the theorem when $\a(\e) \in X_\sm$, as otherwise
  the right hand side of the stated inequality is infinite and the statement is trivial.
  Let us therefore assume that $\a(\e) \in X_\sm$.

  For every $r$, let $J_r := \Fitt^r(\Om_{X/k}) \subset \O_X$. 
  On the one hand, for every finite $n$ we have by \cite[Lemma~8.1]{dFD} that
  \[
    \embdim(\O_{X_n,\a_n}) = (n+1)d_n - \dim(\ov{\{\a_n\}}) + \ord_\a(J_{d_n}),
  \]
  where $d_n = d(\a_n,\Om_{X/k})$ is the Betti number of $\Om_{X/k}$ 
  with respect to $\a_n$ (see \cite[Definition~6.1]{dFD}) and $\ov{\{\a_n\}}$
  the closure of $\a_n$ in $X_n$. 
  On the other hand, since $\a$ is not in an irreducible component
  of $X_\infty$ that is fully contained in $(\Sing(X))_\infty$, 
  we have 
  \[
    \dim(\O_{X_n,\a_n}) \ge (n+1)d - \dim(\ov{\{\a_n\}})
  \]
  for all finite $n$.
  Since for all $n$ large enough we have $d_n = d$, we deduce by 
  \cref{p:embdim=grcodim+dim} that
  $\embcodim(\O_{X_n,\a_n}) \le \ord_\a(J_d)$
  for all $n \gg 1$.
  We conclude by \cref{t:grcodim-of-arcs} that
  $\embcodim(\O_{X_\infty,\a}) \le \ord_\a(J_d)$, 
  as stated. 

  Regarding the last two assertions of the theorem, 
  \eqref{item1:arc-finite-grcodim} follows by the fact that if $X$ is a variety
  then, by definition, $\Jac_X = \Fitt^d(\Om_{X/k})$. 
  As for \eqref{item2:arc-finite-grcodim},
  if $\a(\e) \in X_\sm$ then by \cref{l:X=X^0} we have 
  $\^{\O_{X_\infty,\a}} \isom \^{\O_{X^0_\infty,\a}}$, and hence
  we can apply \eqref{item1:arc-finite-grcodim} to $X^0$; note also that 
  in this case we have $\a \in X^0_\sm$ and hence $\ord_\a(\Jac_{X^0}) < \infty$. 
\end{proof}

We include a proof of the following property, which is well known to experts and is remarked in \cite{Dri}. 

\begin{lemma}
\label{l:X=X^0}
Let $X$ be a scheme of finite type over a field $k$ and $\a \in X_\infty$ an arc with $\a(\e) \in X_\sm$. 
Let $X^0 \subset X$ be the irreducible component containing $\a(\e)$. 
Then $\^{\O_{X_\infty,\a}} \isom \^{\O_{X^0_\infty,\a}}$.
\end{lemma}

\begin{proof}
We may assume that $X = \Spec(R)$ is affine. By abuse of notation we write $\a$ for the map $R \to A[[t]]$. 
Let $\fa := \ker(\a)$. If $(0) = \prod_i \fq_i \subset R$ is a primary decomposition with $\fq_0$
the minimal prime defining $X^0$, then the condition $\a(\e) \in X^0$ translates to 
$\fq_0 \subset \fa$ and $\fq_i \not\subset \fa$ for $i \ne 0$.
Let $A$ be a test-ring, i.e., A is local with maximal ideal $\fm$, residue field $K$ equal to the residue field
of $\a \in X_\infty$,  and $\fm^n = 0$ for some $n \in \N$. Let $\a'$ be any $A$-deformation of $\a$, 
that is, given by a map $R \to A[[t]]$. To prove the lemma, it suffices to show that 
$\fa' := \ker(\a') \supset \fq_0$. We have the commutative diagram
\[
\xymatrix{
R \ar[r]^{\a'} \ar[rd]_{\g'} & A[[t]] \ar[r] \ar[d] & L[[t]] \ar[d] \\
& A(\:\!\!(t)\:\!\!) \ar[r] & L(\:\!\!(t)\:\!\!)
}
\]
where $A(\:\!\!(t)\:\!\!)$ denotes the localization of $A[[t]]$ at the ideal $\fm$. 
Since $A[[t]] \to A(\:\!\!(t)\:\!\!)$ 
is injective, we have $\fa' = \ker(\g')$. Let $f \in \fq_0$. Take any $f_i \in \fq_i \setminus \fa$ for $i \ne 0$. 
Then $g := f \prod_i f_i \in \fa'$. Since $\g'(f_i) \ne 0$ modulo $\fm$, we have that 
$\g'(f_i)$ is a unit. Thus $0 = \g'(g) = \g'(f)u$ where $u$ is a unit, and in particular $f \in \fa'$.
\end{proof}


\begin{theorem}
  \label{t:arc-infinite-embcodim}
  Let $X$ be a scheme of finite type over a perfect field $k$.
  For $\a \in X_\infty$ such that $\a(\e) \in X \setminus X_\sm$, we have
  $\embcodim(\O_{X_\infty,\a}) = \infty$.
\end{theorem}

\begin{proof}
  Note that since $k$ is perfect we have $X \setminus X_\sm = \Sing(X)$, 
  and in particular the condition that $\a(\e) \in X \setminus X_\sm$
  is equivalent to having $\a \in (\Sing(X))_\infty$.

  For every $n \in \N$, let $\p_n \colon X_\infty \to X_n$ be the truncation
  morphism, and let $\a_n := \p_n(\a) \in X_n$. 
  Note that for $n=0$ we have $\a_0 = \a(0)$. 
  Let $L$ and $L_n$ denote the residue fields of $X_\infty$ at $\a$
  and of $X_n$ at $\a_n$.
  By \cite[Lemma~8.3]{dFD} (see also \cite[Remark~7.4]{dFD}), for all $n$
  sufficiently large the differential map
  \[
    (T_\a\p_n)^* \colon \fm_{\a_n}/\fm_{\a_n}^2 
    \otimes_{L_n}L \to \fm_\a/\fm_\a^2
  \]
  has rank at least $(n+1)d(\a)-\dim(\ov{\{\a_n\}})$,
  where 
  \[
    d(\a) := \dim_{k(\a(\e))}(\Om_{X/k} \otimes k(\a(\e))).
  \]
  Then, by \cref{p:grcodim-projection}, we have
  \begin{align*}
    \embcodim(\O_{X_\infty,\a}) &\ge (n+1)d(\a) - \trdeg_k(L_n) - \dim\left(\O_{\ov{\p_n(X_\infty)},\a_n}\right)\\
    & = (n+1)d(\a) - \dim_{\a_n}(\ov{\p_n(X_\infty)})
  \end{align*}
  where $\ov{\p_n(X_\infty)}$ 
  denotes the Zariski closure of 
  $\p_n(X_\infty)$ 
  in $X_n$.

  Since $X$ is of finite type, $X_\infty$ has finitely many irreducible
  components
  (see \cite[Theorem~2.9]{Reg09} and \cite[Corollary~3.16]{NS10}).
  This implies that for $n$ sufficiently large we have
  \[
    \dim_{\a_n}(\ov{\p_n(X_\infty)}) = \max_{C \ni \a} \dim(\ov{\p_n(C)})
  \]
  where the maximum is taken over the irreducible components $C$ of $X_\infty$ that contain $\a$. 

  Let $C$ be one of the irreducible components of $X_\infty$ 
  containing $\a$, let $\b \in C$ be its generic point, and let $Z \subset X$
  be the closure of $\b(\e)$ in $X$. From \cite[Lemma~8.6]{dFD}
  we have
  \[
    \dim(\ov{\p_n(C)})
    \leq
    (n+1) \dim(Z)
    \leq
    (n+1) \dim_{\a(0)}(X).
  \]
  Since $\a(\e)\in \Sing(X)$, we see by the definition
  of $d(\a)$ that $d(\a) > \dim_{\a(0)}(X)$, and therefore
  \[
    \lim_{n \to \infty}\big((n+1)d(\a) - \dim(\ov{\p_n(C)})\big) 
    \geq
    \lim_{n \to \infty}(n+1)\big(d(\a) - \dim_{\a(0)}(X)\big) 
    = \infty.
  \]
  We conclude that $\embcodim(\O_{X_\infty,\a}) = \infty$, as claimed. 
\end{proof}

\begin{corollary}
  \label{c:arc-finite-grcodim-embcodim}
  Let $X$ be a scheme of finite type over a field $k$ and $\a \in X_\infty$.
  Assume that either $k$ has characteristic zero, or $\a \in X_\infty(k)$. 
  Then we have $\a(\e) \in X_\sm$ if and only if 
  $\embcodim(\O_{X_\infty,\a}) < \infty$.
\end{corollary}

\begin{proof}
  If $k$ has characteristic zero, then the \lcnamecref{c:arc-finite-grcodim-embcodim}
  follows by \cref{t:arc-finite-grcodim,t:arc-infinite-embcodim}.

  Let then $k$ be any field, and assume that $\a \in X_\infty(k)$.
  For a field extension $k \subset k'$, we denote
  $X' := X \times_{\Spec(k)} \Spec(k')$ and let $\a' \colon \Spec(k'[[t]]) \to X'$
  be the arc obtained by base change from $\a$. 
  Since a point of $X$ is in the smooth locus if and only if it is geometrically
  regular, we can find a field extension $k \subset k'$ such that
  $\a'$ is not a regular point of $X'$. By faithfully flat descent of regularity,  
  we can replace $k'$ with a larger
  field extension and assume without loss of generality that $k'$ is perfect. 
  Note that $X'_\infty \isom X_\infty \times_{\Spec(k)} \Spec(k')$, and hence
  $\O_{X'_\infty,\a'} \isom \O_{X_\infty,\a} \otimes_k k'$. 
  Then, by \cref{p:grcodim-field-change}, we have
  \[
    \embcodim(\O_{X'_\infty,\a'}) = \embcodim(\O_{X_\infty,\a}).
  \]
  This reduces to the case of perfect fields, where the result follows
  from again by \cref{t:arc-finite-grcodim,t:arc-infinite-embcodim}.
\end{proof}


\section{On Drinfeld--Grinberg--Kazhdan's theorem}

\label{s:DGK}

\cref{t:arc-finite-grcodim} can be seen as a finiteness statement for
singularities of the arc space at arcs that are not fully contained in the
singular locus. One of the first major results in this direction is the theorem
of Drinfeld, Grinberg, and Kazhdan, which we will state below in its version in
\cite{Dri}.  Recall that for any equicharacteristic local ring $(A,\fm,k)$ a
DGK decomposition is an isomorphism $\^A\isom k[[t_i \mid i\in I]]/\fa$, where
$\fa$ is an ideal of finite polynomial definition.

\begin{theorem}[\protect{\cite[Theorem~2.1]{GK00}, \cite[Theorem~0.1]{Dri}}]
  \label{t:DGK}
  Let $X$ be a scheme of finite type over a field $k$, and let $\a \in
  X_\infty(k)$. If $\a(\e) \in X_\sm$, then the local ring $\O_{X_\infty,\a}$
  admits a DGK decomposition. 
\end{theorem}

As mentioned in \cref{r:DGK-and-finite-def}, any DGK decomposition of
$\O_{X_\infty,\a}$ induces an isomorphism of formal schemes
\[
  \^{X_{\infty,\a}} \isom \^{Z_z} \hat\times \D^\N,
\]
with $Z$ a scheme of finite type over $k$, $z\in Z(k)$ and $\D=\Spf(k[[t]])$. The formal scheme $\^{Z_z}$ is
often referred to as a \emph{formal model} for $\a$. While $\^{Z_z}$ is not unique, there
exists a unique minimal one in the following sense.

\begin{theorem}[\protect{\cite[Theorem 7.1]{BS17b}, \cite[Theorem 1.2]{BS19a}}]
  Let $\^{Z_z}$ and $\^{W_w}$ be two formal models for $\a$ which are
  \emph{indecomposable}, i.e., they are not of the form $\mathcal{Y} \hat\times \D$
  with $\mathcal{Y}$ a formal scheme. Then $\^{Z_z}\isom \^{W_w}$.
\end{theorem}

\begin{definition}
  The indecomposable formal model of $\a$ is called the \emph{minimal formal
  model} and denoted by $\mathcal{Z}^{\min}_\a$.
\end{definition}

Note that the formal model provided by \cref{t:DGK} is not minimal in general,
which we will see in \cref{s:Drinfeld-model}.

Combining \cref{t:DGK} with the results of this paper, we obtain the next
result which provides a characterization of $k$-rational arcs admitting a DGK
decomposition. The result also gives an explicit bound for the embedding
codimension; we should stress that such bound does not follow from the proofs
in \cite{GK00,Dri,Dri18}.

\begin{theorem}
  \label{t:DGK+converse}
  Let $X$ be a scheme of finite type over a field $k$. For any $\a \in
  X_\infty(k)$, the following are equivalent:
  \begin{enumerate}
    \item
    \label{item1:DGK+converse}
      $\a(\e) \in X_\sm$.
    \item
    \label{item2:DGK+converse}
      $\O_{X_\infty,\a}$ admits a DGK decomposition.
    \item
    \label{item3:DGK+converse}
      $\O_{X_\infty,\a}$ admits a weak DGK decomposition.
    \item
    \label{item4:DGK+converse}
      $\embcodim(\O_{X_\infty,\a})<\infty$.
  \end{enumerate}
  Moreover, if $k$ is perfect and $\a(\e) \in X_\sm$, then 
  \[
  \embcodim(\O_{X_\infty,\a})\leq \ord_\a(\Jac_{X^0})
  \]
  where $X^0 \subset X$ is the irreducible component containing $\a(\e)$. 
\end{theorem}

\begin{proof}
  The implication $\eqref{item1:DGK+converse} \Rightarrow \eqref{item2:DGK+converse}$ is \cref{t:DGK},
  the implication $\eqref{item2:DGK+converse} \Rightarrow \eqref{item3:DGK+converse}$ is obvious,
  the implication $\eqref{item3:DGK+converse} \Rightarrow \eqref{item4:DGK+converse}$ follows from \cref{p:finite-def-finite-ht},
  and finally \cref{c:arc-finite-grcodim-embcodim} gives the implication 
  $\eqref{item4:DGK+converse} \Rightarrow \eqref{item1:DGK+converse}$.
  The last statement follows from \cref{t:arc-finite-grcodim}.
\end{proof}

\begin{example}
  \label{e:first-DGK-example}
  Let $X$ be the hypersurface defined by $x_0x_{n+1}+f(x_1,\ldots,x_n)=0$ and
  $\a\in X(k)$ the arc given by $(t,0,\ldots,0)\in k[[t]]^{n+2}$. Assume
  further that the hypersurface $H \subset \A^n$ given by $f(x_1,\ldots,x_n) =
  0$ has a singularity at $0$. Then, as shown in \cite{Dri}, a DGK
  decomposition for $\O_{X_\infty,\a}$ is given by
  \[
    \^{\O_{X_\infty,\a}}
    \isom
    k[[x_1,\ldots,x_n]]/f(x_1,\ldots,x_n)
    \,\cotimes_k\,
    k[[t_i \mid i\in\N]].
  \]
  The singularity of $\a$ is thus again given by $H$ and
  $\embcodim(\O_{X_\infty,\a})=1$. On the other hand, the order of $\a$ with
  respect to the Jacobian ideal $\Jac_X$ is $1$, hence the bound in
  \cref{t:DGK+converse} is sharp in this case.
\end{example}

\begin{example}
  \label{e:second-DGK-example}
  Similarly as in the previous example, let $X$ be the hypersurface defined by
  $x_0x_{n+1} + f(x_1,\dots,x_n) = 0$ where $f$ is a polynomial of multiplicity
  2, and take this time $\a \in X_\infty(k)$ to be the arc given by
  $(t^m,0,\dots,0) \in k[[t]]^{n+2}$. Denoting by $g^{(j)}$ the $j$-th
  Hasse--Schmidt derivative of an element $g \in k[x_0,\dots,x_{n+1}]$ and
  setting for short $I = \{0,1,\dots,n+1\}$ and $J = \Z_{\ge 0}$, $X_\infty$ is
  defined by the ideal $\fa = ((x_0x_{n+1} + f)^{(j)} \mid j \in J)$ of $P :=
  k[x_i^{(j)} \mid (i,j) \in I \times J]$, see \cite{Voj07}. Let $\fm \subset P$ be the maximal
  ideal at $\a$. Since $x_0^{(m)}$ is a unit in the local ring $P_\fm$, we see
  that the ideal $\ini(\fa P_\fm)$ is generated by the elements $x_{n+1}^{(j)}$
  for $j \in J$, and $\ini(f^{(l)})$ for $0 \le l \le m-1$. As long as $f$ is
  chosen so that $\ini(f^{(l)})$, for $0 \le l \le m-1$, form a regular
  sequence (e.g., $f = x_1x_2$ would work), we get that $\O_{X_\infty,\a}$ has
  embedding codimension $m$. Since clearly the order of $\a$ with respect to
  $\Jac_X$ is also $m$, this shows that the bound in \cref{t:DGK+converse} is
  sharp for all possible orders of the arc with the Jacobian ideal of $X$. 
\end{example}

Let us mention here the following consequence of \cref{t:DGK+converse}, which
implies that the local rings of closed arcs in the arc space provide plenty of
examples of non-Noetherian rings for which their embedding codimension agrees
with their formal embedding codimension.

\begin{corollary}
  \label{c:emb=femb}
  If $X$ is a scheme of finite type over a field $k$, then
  \[
    \embcodim(\O_{X_\infty,\a}) = \fembcodim(\O_{X_\infty,\a})
  \]
  for every $\a \in X_\infty(k)$. If $k$ is perfect, then the same holds for
  all constructible points $\a \in X_\infty$ with $\a(\e) \in X_\sm$.
\end{corollary}

\begin{proof}
  Assume first that $\a \in X_\infty(k)$. If $\a(\e) \in X_\sm$ then the
  equality follows by \cref{t:DGK,p:DGK-grcodim}. If $\a(\e) \in X \setminus
  X_\sm$, then we have $\embcodim(\O_{X_\infty,\a}) = \infty$ by
  \cref{c:arc-finite-grcodim-embcodim}, and we conclude by
  \cref{t:grcodim-embcodim}. Suppose now that $\a \in X_\infty$ is a
  constructible point with $\a(\e) \in X_\sm$. By \cite[Theorem~10.8]{dFD},
  $\O_{X_\infty,\a}$ has finite embedding dimension, and hence the assertion
  follows by \cref{p:grcodim-embcodim:finite-embdim}.
\end{proof}

We now state the following application of \cref{t:general-proj,t:DGK}, which
says that a generic projection of the base scheme induces an efficient DGK
decomposition at an arc that is not contained in the singular locus.

\begin{theorem}
  \label{c:proj-eff-formal-emb}	
  Let $X \subset \A^N$ be an affine scheme of finite type over a perfect field $k$, $\a\in
  X_\infty(k)$ with $\a(\e)\in X_\sm$ and $d=\dim_{\a(\e)}(X)$. Let
  $f\colon X\to Y:=\A^d$ be the map induced by a
  general linear projection $\A^N \to \A^d$, and let $\b:=f_\infty(\a)$.
  Then the associated map
  \[
    \varphi\colon \O_{Y_\infty,\b} \to \O_{X_\infty,\a},
  \]
  gives an efficient formal embedding of $\O_{X_\infty,\a}$. 
  Moreover, if $k$ is infinite, then there exist formal
  coordinates $u_i\in \^{\O_{Y_\infty,\b}}$, $i\in \N$, such that
  $\ker(\^\varphi)$ is generated by finitely many polynomials in $u_i$, hence
  $\^\varphi$ induces an efficient DGK decomposition.
\end{theorem}

\begin{proof}
  The first part follows from \cref{t:general-proj} together with the fact that
  $\O_{Y_\infty,\b}$ is formally smooth over $k$. Regarding the second assertion,
  we know by \cref{t:DGK} that the map $\varphi$ induces a surjection
  \[
    \psi\colon \^{\O_{Y_\infty,\b}} \to \^{\O_{Z,z}} \cotimes_k k[[t_i \mid i\in \N]]
  \]
  with $Z$ a scheme of finite type over $k$ and $z\in Z(k)$. If $Z$ is smooth
  at $z$ then there is nothing to show. Otherwise, by \cref{t:min-zariski-emb},
  we may assume that $Z\subset \A^n$, where $n=\embdim(\O_{Z,z})$. Since $\psi$
  induces an isomorphism of continuous cotangent spaces the statement follows
  from \cref{p:emb-indep-basis}.
\end{proof}

The next example illustrates in concrete terms the content of
\cref{c:proj-eff-formal-emb} when $X$ is a hypersurface in an affine space,
where the existence of the efficient formal embedding as in \lcnamecref{c:proj-eff-formal-emb}
can be verified directly from the equations.

\begin{example}
  Let $f\in k[x_1,\ldots,x_n,y]$ and $X$ be the hypersurface defined by $f$.
  For the sake of convenience, we will write $x=(x_1,\ldots,x_n)$. Let
  $\a=(x(t),y(t))$ be an arc on $X$ such that
  $\ord_\a(\Jac_X)=\ord_t(\frac{\partial f}{\partial y}(x(t),y(t)))=d>0$. We
  write $x(t)=\sum_j a^{(j)} t^j$ and $y(t)=\sum_j b^{(j)} t^j$; note that
  $a^{(j)}=(a_1^{(j)},\ldots,a_n^{(j)})$. Let
  $D=(D_p)_{p\geq0}$ be the universal Hasse--Schmidt derivation on
  $k[x^{(j)},y^{(j)} \mid j\geq0]$, where $x^{(j)}=(x_1^{(j)},\ldots,x_n^{(j)})$. 
  Then $X_\infty=\Spec(R_\infty)$, where
  \[
    R_\infty=k[x^{(j)},y^{(j)} \mid j\geq0]/(f^{(p)} \mid p\geq0),
  \]
  with $f^{(p)}:=D_p(f)$. Note that $f^{(p)}$ depends only on $x^{(j)},
  y^{(j)}$ for $j\leq p$. The arc $\a$ then corresponds to the ideal $\fm_\a$
  of $R_\infty$ given by
  \[
    \fm_\a=(x^{(j)}-a^{(j)},y^{(j)}-b^{(j)} \mid j\geq0).
  \]
  Setting $\~{{f}^{(p)}}(x^{(j)},y^{(j)}):=f^{(p)}(x^{(j)}+a^{(j)},y^{(j)}+b^{(j)})$, we get that
  \[
    \O_{X_\infty,\a}\isom k[x^{(j)},y^{(j)}]_{(x^{(j)},y^{(j)})}/(\~{{f}^{(p)}}).
  \]
  We are going to make use of the following explicit formula from
  \cite[Section~5]{dFD}:
  \[
    \frac{\partial f^{(p)}}{\partial y^{(q)}}
    =
    D_{p-q}\left(\frac{\partial f}{\partial y}\right),\quad q\leq p.
  \]
  The condition $\ord_t(\frac{\partial f}{\partial y}(x(t),y(t)))=d$ implies
  that, for $p\geq d$,
  \[
    \frac{\partial \~{{f}^{(p)}}}{\partial y^{(q)}}(0,0)
    =
    \frac{\partial f^{(p)}}{\partial y^{(q)}}(a,b)
    =
    D_{p-q}\left(\frac{\partial f}{\partial y}\right)(a,b)
    \begin{cases}
      =0, & p-d<q\leq p \\
      \neq 0, & q=p-d.
    \end{cases}
  \]
  Now the above implies that the initial forms of $f^{(p)}$, for $p\geq d$, can
  be written as
  \[
    \ini(\~{f^{(d+i)}})=y^{(i)}+g^{(d+i)}
  \]
  where $g^{(d)} \in k[x^{(j)} \mid j\leq d]$ and, for $i > 0$, $g^{(d+i)}
  \in k[x^{(j)},y^{(l)} \mid j\leq d+i, l<i]$.
  In particular, the elements $x^{(j)}$ and $\~{f^{(d+j)}}$, for $j\geq 0$,
  give formal coordinates in $k[[x^{(j)},z^{(j)} \mid j\geq0]]$, hence the
  map
  \[
    \varphi\colon
        k[[x^{(j)},z^{(j)} \mid j\geq0]]
        \to
        k[[x^{(j)},y^{(j)} \mid j\geq0]],
    \quad
    x^{(j)}\mapsto x^{(j)},\; z^{(j)}\mapsto \~{f^{(j+d)}}
  \]
  is an isomorphism. Write $h_i:=\varphi^{-1}(f^{(i)})$ and
  $\fa:=(\bar{h}_0,\ldots,\bar{h}_{d-1})$, where $\bar{h}_i$ is obtained from
  $h_i$ by setting $z^{(j)}=0$ for all $j\geq0$. Then we get that
  \[
    \^{\O_{X_\infty,\a}}\isom k[[x^{(j)} \mid j\geq0]]/\^\fa.
  \]
  Observe that the map $k[[x^{(j)} \mid j\geq0]] \to \^{\O_{X_\infty,\a}}$ is
  the efficient formal embedding from \cref{c:proj-eff-formal-emb} with respect
  to the projection $(x,y)\mapsto x$. However, this isomorphism does not
  induce a DGK decomposition a priori since the ideal $\^\fa$ is not
  necessarily of finite polynomial definition with respect to the variables
  $x^{(j)}$. 
\end{example}


\section{Efficient embedding of the Drinfeld model}

\label{s:Drinfeld-model}

It is useful to compare the formal embedding given by
\cref{c:proj-eff-formal-emb} to the one provided by the
Drinfeld--Grinberg--Kazhdan theorem. This comparison is done below in
\cref{p:dgk-comparison}. We first need to recall the construction of the
Drinfeld models.

Let $X \subset \A^N$ be an affine scheme of finite type over a field $k$,
consider a $k$-rational arc $\alpha \in X_\infty(k)$ such that $\alpha(\eta) \in
X_\sm$, and let $d := \dim_{\alpha(\eta)}(X)$ and $c := N-d$. Let $X^0$ be the
irreducible component of $X$ containing $\a(\e)$ (note that $d = \dim X^0$), 
and let $X' \supset X^0$ be the complete intersection scheme
defined by the vanishing of $c$ general linear combinations $p_1, \ldots, p_c$ of 
a set of generators of the ideal of $X^0$ in $\A^N$. 
As explained in \cite{Dri}, the respective inclusions induce isomorphisms
$\^{\O_{X_\infty,\a}} \isom \^{\O_{X^0_\infty,\a}} \isom \^{\O_{X'_\infty,\a}}$
(detailed proofs are given in \cref{l:X=X^0} and \cite[Section~4.2]{BS17b}).
Pick coordinates $x_1, \ldots, x_d, y_1, \ldots, y_c$ in the ambient affine space
$\A^N$. For a general choice of such coordinates, we can assume that
\[
    \ord_\a\left(
        \det\left(
            \frac%
                {\partial (p_1,\ldots,p_c)}%
                {\partial (y_1,\ldots,y_c)}%
        \right)
    \right)
    = \ord_\a(\Jac_{X'})
    = \ord_\a(\Jac_{X^0})
    =: e < \infty.
\]

Drinfeld defines a specific formal model for $\^{\mathcal O_{X_\infty,\a}}$
depending only on the choices of the coordinates $x_i, y_j$, the equations
$p_l$, and the order of contact $e$. Concretely, consider the affine space
$\A^{m}$ where $m = e(1+2d+c)$. 
We denote by $R[t]_{<n}$ the space of polynomials of degree
$< n$ with coefficients in $R$. 
Denoting by $Q_n$ the scheme representing the functor
$R \mapsto t^n + R[t]_{<n}$, the space of monic polynomials of degree
$n$ with coefficients in $R$, we identify $\A^m$ with the product
$Q_e \times \A^d_{2e-1} \times \A^c_{e-1}$. Under this identification, 
a $k$-rational point of $\A^m$ corresponds to a triple 
\[
    (q(t), \bar x(t), \bar y(t)) 
    \in
    (t^e + k[t]_{< e})
    \times
    (k[t]_{< 2e})^d
    \times
    (k[t]_{< e})^c.
\]
In particular, coordinates in $\A^m$ take the form $q^{(n)}, \bar x_i^{(n)}, \bar y_j^{(n)}$.
Consider the conditions
\begin{equation}
\label{eq:dri-model-equations}
\begin{array}{rccl}
    p_{1}(\bar x(t), \bar y(t)) \;\;\equiv\;\; \cdots \;\;\equiv\;\; p_{c}(\bar x(t), \bar y(t)) 
    &\equiv &0 &\mod q(t), \\[1em]
    \det\left(
        \dfrac%
            {\partial (p_1,\ldots,p_c)}%
            {\partial (y_1,\ldots,y_c)}%
        (\bar x(t), \bar y(t)) 
    \right)
    &\equiv &0 &\mod q(t), \\[1.6em]
    \adj\left(
        \dfrac%
            {\partial (p_1,\ldots,p_c)}%
            {\partial (y_1,\ldots,y_c)}%
        (\bar x(t), \bar y(t)) 
    \right)
    \begin{pmatrix}
        p_1(\bar x(t), \bar y(t))\\
        \vdots \\
        p_c(\bar x(t), \bar y(t))
    \end{pmatrix}
    &\equiv 
    &
    \begin{pmatrix}
        0 \\ \vdots \\ 0
    \end{pmatrix}
    &\mod q(t)^2.
    \\[2em]
\end{array}
\end{equation}
Here $\adj(B)$ denotes the classical adjoint of a matrix $B$. As explained in
\cite{Dri} and \cite[Sections~3.3~and~3.4]{BS17b}, the conditions in
\eqref{eq:dri-model-equations} are polynomial in the coefficients of $q(t),
\bar x(t), \bar y(t)$, and therefore they define a finite type subscheme $Z
\subset \A^m$.

Write the arc $\a$ in the coordinates $(x,y)$ of $\A^N$ as $\a = (a(t), b(t))$
where $a(t) \in k[[t]]^d$ and $b(t) \in k[[t]]^c$. To $\a$ we associate the
point $z = (t^e, \bar a(t), \bar b(t)) \in Z$ given by
\begin{equation}
    \label{eq:dri-z}
    \bar a(t) \equiv a(t) \mod t^{2e},
    \qquad
    \bar b(t) \equiv b(t) \mod t^{e}.
\end{equation}
It is shown in \cite{Dri} that $\^{Z_z}$ gives a (finite-dimensional)
formal model for $\a$, that is,
\begin{equation}
    \label{eq:dri-iso}
    \^{X_{\infty,\a}} \isom \^{Z_z} \hat\times \D^\N.
\end{equation}

The isomorphism in \eqref{eq:dri-iso} can be expressed somewhat explicitly in
coordinates. We identify $\A^d_\infty$ with an infinite-dimensional
affine space $\A^\N$, and we use the notation $\x(t)$ for points in $\A^\N$.
Hence coordinates in $\A^\N$ take the form $\x_i^{(n)}$, with $1 \le i \le d$ 
and $n \ge 0$. The disk $\D^\N$
appearing in \eqref{eq:dri-iso} is the formal neighborhood of $c(t)$ in $\A^\N$,
where $c(t) := t^{-2e}(a(t))_{\geq 2e}$ is the truncation of $a(t)$ to degrees
$\geq 2e$ divided by $t^{2e}$. Summarizing, we have described coordinates
$(x(t),y(t))$ in $\^{X_{\infty,\a}}$ and coordinates $(q(t), \bar x(t), \bar
y(t), \x(t))$ in $\^{Z_z} \hat\times \D^\N$. As explained in \cite{Dri}, the
isomorphism in \eqref{eq:dri-iso} gives the relation
\[
    x(t) = q(t)^2\x(t) + \bar x(t), 
\]
and we have
\begin{equation}
\label{eq:dri-proj}
    \bar x(t) \equiv  x(t) \mod q(t)^2, 
    \qquad
    \bar y(t) \equiv  y(t) \mod q(t).
\end{equation}
We emphasize that these relations only hold at the level of formal
neighborhoods. 

Notice that the point $z \in Z$ depends on the arc $\a$, but the scheme $Z$
only depends on the choices of the coordinates $x_i, y_j$, the equations $p_l$,
and the order of contact $e$. The choice of coordinates $x_i,y_j$
also determines the linear projection $\A^N \to \A^d$ given by $(x,y) \mapsto x$,  
and hence the induced map $f \colon X \to \A^d$.

\begin{definition}
With the above notation, we say that $(Z,z)$ is a \emph{Drinfeld model} of $X_\infty$ at $\alpha$, 
and that it is \emph{compatible} with $f$. 
\end{definition}

We are now ready to state and prove our comparison theorem.

\begin{theorem}
  \label{p:dgk-comparison}	
  Let $X \subset \A^N$ be an affine scheme of finite type over a perfect field
  $k$, let $\a\in X_\infty(k)$ with $\a(\e)\in X_\sm$, and let
  $d:=\dim_{\a(\e)}(X)$. Let $f \colon X \to Y := \A^d$ be induced by a general
  linear projection $\A^N \to \A^d$ and $\b :=f_\infty(\a)$.
  Let $(Z,z)$ be a Drinfeld model compatible
  with $f$, and let $\^{X_{\infty,\a}} \isom \^{Z_z}
  \hat\times \D^\N$ be the corresponding DGK decomposition.
  \begin{enumerate}
    \item
    \label{item1:dgk-comparison}
      The composition map 
      \[
        \^{Z_z} \hat\times \D^\N
        \xrightarrow{\sim}
        \^{X_{\infty,\a} }
        \inj
        \^{Y_{\infty,\b}}
      \] 
      is the completion of a morphism
      $g\colon Z \times \A^\N \to Y_\infty$. 
    \item  
    \label{item2:dgk-comparison}
      If $X^0 \subset X$ is the irreducible component of containing $\a(\e)$
      and $e := \ord_\a(\Jac_{X^0})$, then the composition map 
      \[
        \^{Z_z} \inj \^{Z_z} \hat\times \D^\N \xrightarrow{\sim} 
        \^{X_{\infty,\a} }\inj \^{Y_{\infty,\b}} \surj \^{Y_{2e-1,\b_{2e-1}}}
      \] 
      is an efficient formal embedding.
      Moreover, at the level of associated graded rings, we have that
      \[
        \gr(\O_{Z,z})=\im(\gr(\O_{Y_{2e-1},\b_{2e-1}})\to \gr(\O_{X_\infty,\a})).
      \]
  \end{enumerate}
\end{theorem}

\begin{proof}
We use the notation introduced at the beginning of the section. In particular: we have coordinates
$(x, y)$ in $\A^N$ such that the projection $\A^N \to Y$ is given by $(x,y)
\mapsto x$; we write $\a = (a(t), b(t))$, so $\b = f_\infty(\a) = a(t)$; we have a
space $\A^m$ with coordinates $(q(t), \bar x(t), \bar y(t))$, and $Z \subset
\A^m$ is defined by the conditions in \eqref{eq:dri-model-equations}; the point
$z \in Z$ is given by $z = (t^e, \bar a(x), \bar b(x))$ as in \eqref{eq:dri-z};
the formal scheme $\^{Z_z} \hat\times \D^\N$ is contained in the completion of
$\A^m \times \A^\N$ at $(z,c(t))$ where $c(t) = t^{-2e}(a(t))_{\geq 2e}$, 
and the coordinates in this affine space
have the form $(q(t), \bar x(t), \bar y(t), \xi(t))$.

We define a map $\A^m \times \A^\N \to Y_\infty$ via
\[
    (q(t), \bar x(t), \bar y(t), \xi(t))
    \;\mapsto\;
    q(t)^2\xi(t) + \bar x(t),
\]
and we let $g \colon Z \times \A^\N \to Y_\infty$ be the restriction. It is
clear from the discussion surrounding \eqref{eq:dri-proj} that the completion of
$g$ gives the composition of $\^{Z_z} \hat\times \D^\N \xrightarrow{\sim}
\^{X_{\infty,\a} } \inj \^{Y_{\infty,\b}}$, and the first statement of the
\lcnamecref{p:dgk-comparison} follows.

We compute the tangent map of $g$ explicitly. 
With a small abuse of notation where coordinates of elements and
coordinate functions are written in the same way,
we denote a tangent vector on
$Z \times \A^\N$ based at a point $(z,c(t)) = (t^e, \bar a(t), \bar b(t), c(t))$ by
\[
    (\,
        t^e + dq(t) \ep,
        \; \bar x(t) + d{\bar x}(t) \ep,
        \; \bar y(t) + d{\bar y}(t) \ep,
        \; c(t) + d\xi(t) \ep
    \,)
\]
where 
$dq(t) \in k[t]_{<e}$,
$d{\bar x}(t) \in (k[t]_{<2e})^d$,
$d{\bar y}(t) \in (k[t]_{<e})^c$,
$d\xi(t) \in k[t]^d$,
and $\ep^2 = 0$.
The image of such a tangent vector under $g$ is given by
\begin{align*}
    & \big(t^e + dq(t)\ep\big)^2
    \big(c(t) + d\xi(t) \ep\big)
    +
    \big(\bar x(t) + d{\bar x}(t) \ep\big) \\
    & \qquad = 
    \big(t^{2e} c(t) + \bar x(t)\big)
    +
    \big(d{\bar x}(t) + d\xi(t)t^{2e} + 2c(t)dq(t)t^e\big) \ep \\
    & \qquad = 
    a(t)
    +
    \big(d{\bar x}(t) + d\xi(t)t^{2e} + 2c(t)dq(t)t^e\big) \ep.
\end{align*}
In other words, the tangent map of $g$ at $(z,c(t))$ is given by
\[
    dx(t) 
    = 
    d{\bar x}(t) + d\xi(t)t^{2e} + 2c(t)dq(t)t^e,
\]
or, in coordinates (recall that $c(t) = t^{-2e}(a(t))_{\geq 2e}$), by
\[
    dx_i^{(n)} =
    \begin{cases}
        d{\bar x}_i^{(n)} 
        & \text{if $n<e$,} \\
        d{\bar x}_i^{(n)} + 2 \sum_{k+l=n-e} a_i^{(k+2e)}dq^{(l)}
        & \text{if $e \leq n<2e$,} \\
        d\xi_i^{(n-2e)} + 2 \sum_{k+l=n-e} a_i^{(k+2e)}dq^{(l)}
        & \text{if $n \geq 2e$.}
    \end{cases}
\]
From this we see that $g$ induces a surjective map between associated graded
rings
\[
    \varphi \colon
    k\!\left[\,
        dx^{(n)}_i
        \,\middle|\,
        \begin{smallmatrix}
            n\in\N
            \\
            1 \leq i \leq d
        \end{smallmatrix}
    \,\right]
    \isom
    \gr(\O_{Y_\infty,\b})
    \longrightarrow
    \gr(\O_{Z,z})
    \otimes_k 
    k\!\left[\,
        d\xi^{(m)}_i
        \middle|
        \begin{smallmatrix}
            m\in\N
            \\
            1 \leq i \leq d
        \end{smallmatrix}
    \right].
\]
We have a commutative diagram
\[
\xymatrix{
    k\!\left[\,
        dx^{(n)}_i
        \,\middle|\,
        \substack{
            n \in \N
            \\
            1 \leq i \leq d
        }
    \,\right]
    \ar[r]^-\varphi
    &
    \gr(\O_{Z,z})
    \otimes_k
    k\!\left[\,
        d\xi^{(m)}_i
        \,\middle|\,
        \substack{
            m\in\N
            \\
            1 \leq i \leq d
        }
    \,\right]
    \ar[d]^-\lambda
    \\    
    k\!\left[\,
        dx^{(n)}_i
        \,\middle|\,
        \substack{
            n < 2e
            \\
            1 \leq i \leq d
        }
    \,\right]
    \otimes_k
    k\!\left[\,
        dx^{(n)}_i
        \,\middle|\,
        \substack{
            n \ge 2e
            \\
            1 \leq i \leq d
        }
    \,\right]
    \ar@{=}[u]
    \ar[r]^-{\ff \otimes_k \m}
    & 
    \gr(\O_{Z,z})
    \otimes_k
    k\!\left[\,
        d\xi^{(m)}_i
        \,\middle|\,
        \substack{
            m\in\N
            \\
            1 \leq i \leq d
        }
    \,\right]
}
\]
where $\lambda$ is the $\gr(\O_{Z,z})$-linear map given by 
\[
    d\xi^{(m)}_{i}
    \;\mapsto\;
    d\xi^{(m)}_{i} - 2\sum_{k+l=m-e} a_i^{(k+2e)}dq^{(l)}, 
\]
$\m$ is given by $dx_i^{(n)} \mapsto d\x_i^{(n-2e)}$ for $n \ge 2e$,
and $\psi$ agrees with the map $\gr(\O_{Y_{2e-1},\b_{2e-1}}) \to \gr(\O_{Z,z})$
induced by the composition 
\[
    Z\to Z\times\{c(t)\}\hookrightarrow Z\times \A^\N \to Y_\infty \to Y_{2e-1},
\]
which is given by
\[
    (\bar{x}(t),\bar{y}(t),q(t))
    \;\mapsto\;
    \bar{x}(t)
    +
    t^{-2e}(a(t))_{\geq 2e}\,q(t)^2 \mod t^{2e}.
\]

The map $\lambda$ is invertible and, by \cref{t:general-proj}, the map $\varphi$ is surjective, 
thus $\psi$ is surjective as well. This implies that
\[
    \gr(\O_{Z,z})
    =
    \im(\gr(\O_{Y_{2e-1},\b_{2e-1}})\to \gr(\O_{X_\infty,\a})),
\]
and hence the last assertion follows.
For the first part of \eqref{item2:dgk-comparison}, the fact that $\psi$ is surjective implies that
the map induced on completions $\^{O_{Y_{2e-1},\b_{2e-1}}}\to \^{\O_{Z,z}}$
is surjective as well. The fact that this is an efficient embedding follows from the injectivity of
the corresponding tangent map.
\end{proof}


\section{Applications to Mather--Jacobian discrepancies}

Throughout this section, let $X$ be a variety over a field $k$ of characteristic zero. 

Given a prime divisor $E$ on a normal birational model $f \colon Y \to X$, we define 
the \emph{Mather discrepancy} $\^k_E :=  \ord_E(\Jac_f)$ and the 
\emph{Mather--Jacobian discrepancy} (or simply \emph{Jacobian discrepancy}) 
$k^\MJ_E := \^k_E - \ord_E(\Jac_X)$ of $E$ over $X$.
Note that these definitions only depend on the valuation $\ord_E$ defined by $E$ and not
by the particular model chosen. The definition extends to any divisorial valuation $v = q \ord_E$, 
where $q$ is a positive integer, by setting $\^k_v := q\^k_E$ and  $k^\MJ_v := q k^\MJ_E$.
When $X$ is smooth, both discrepancies agree with the usual discrepancy of $E$ over $X$. 
We say that $X$ is \emph{MJ-terminal} if $k^\MJ_E > 0$ whenever $E$ is exceptional over $X$. 
As proved in \cite{Ish13,dFD14}, this condition is equivalent to the condition that 
if $X \subset Y$ is a closed embedding with $Y$ smooth and $c = \codim(X,Y)$, then 
for any closed subset $T \subsetneq X$ the pair $(Y,cX)$ has
minimal log discrepancy $\mld_T(Y,cX) > 1$. 
We refer to \cite{dFEI08,Ish13,dFD14,EI15} for general studies 
related to these invariants.

The results of this article, together with a theorem from \cite{dFD}, yields
a new proof of the following theorem of Mourtada and Reguera. 

\begin{theorem}[\protect{\cite[Theorem~4.1]{MR18}}]
\label{t:MR}
With the above notation, let $\a \in X_\infty$ be the maximal arc defining 
a given divisorial valuation $q\ord_E$ (i.e., such that $\ord_\a = q\ord_E$).
Then $\dim(\^{\O_{X_\infty,\a}}) \ge q (k_E^\MJ + 1)$.
\end{theorem}

\begin{proof}
We have $\embdim(\O_{X_\infty,\a}) = q(\^k_E+1)$ by \cite[Theorem~11.4]{dFD}, 
and \cref{t:arc-finite-grcodim} gives us $\embcodim(\O_{X_\infty,\a}) \le \ord_\a(\Jac_X)$. 
It follows then by \cref{p:embdim=embcodim+dim,p:grcodim-embcodim:finite-embdim} that
\[
\dim(\^{\O_{X_\infty,\a}})
=
\embdim(\O_{X_\infty,\a}) - \embcodim(\O_{X_\infty,\a}) \ge q (k_E^\MJ + 1).
\qedhere
\]
\end{proof}

Assume now that $X$ is an affine toric variety. To fix notation, let $T$ be an algebraic $k$-torus, 
$N := \Hom_k({\mathbb G}_m,T)$, $M := \Hom_\Z(M,\Z)$, 
$\s \subset N_\R$ a rational convex cone, and $X := \Spec k[\s^\vee \cap M]$.
Note that every $v \in \s \cap N$ defines a $T$-invariant divisorial valuation on $X$.

In their recent article \cite{BS19b}, 
Bourqui and Sebag study DGK decompositions of $X_\infty$ at arcs that are not fully contained
in the $T$-invariant divisor of $X$. The focus is on the open set
$X_\infty^\o \subset X_\infty$ consisting of those arcs whose generic point is in $T$.
They prove that for any $\a \in X_\infty^\o$, the local ring $\O_{X_\infty,\a}$
only depends on the associated valuation $\ord_\a$, 
and in particular so does the minimal formal model
\cite[Corollary~3.3]{BS19b}.
In particular, if we set
\[
X_{\infty,v}^\o := \{ \a \in X_\infty^\o \mid \ord_\a = v \},
\]
then we can denote by $\mathcal{Z}^{\min}_v$ the minimal formal
model of $X_\infty$ at any arc $\a \in X_{\infty,v}^\o$.

The next theorem is one of the main results of \cite{BS19b}.
A similar, more general property is proved for elements $v$
satisfying a certain property called $\cP_v$; we refer to the original source for the precise statement.

\begin{theorem}[\protect{\cite[Corollary~6.4]{BS19b}}]
\label{t:BS}
With the above notation, if 
$v$ is indecomposable in $\s \cap N$, then the associated minimal formal model $\mathcal{Z}^{\min}_v$
has $\dim(\mathcal{Z}^{\min}_v) = 0$ and $\embdim(\mathcal{Z}^{\min}_v) = \^k_v$.
\end{theorem}

Indecomposable elements $v \in \s \cap N$ are characterized by the property that their
centers on any resolution of singularity of $X$ are irreducible components of codimension 1
of the exceptional locus, see \cite[Theorem~2.7]{BS19b}. In the terminology of the Nash problem, 
these form a particular class of \emph{essential valuations}. 
By combining the above \lcnamecref{t:BS} with our results, we obtain the following corollary. 

\begin{corollary}
\label{c:toric-k^MJ}
Let $X = \Spec k[\s^\vee \cap M]$ be an affine toric variety. 
\begin{enumerate}
\item
\label{item1:toric-k^MJ}
For any indecomposable element $v \in \s \cap N$, we have $k^\MJ_v \le 0$. 
\item
\label{item2:toric-k^MJ}
If $X$ is singular and $\Q$-factorial, then $X$ is not MJ-terminal. 
\end{enumerate}
\end{corollary}

\begin{proof}
Part \eqref{item1:toric-k^MJ} follows immediately from \cref{t:BS}
and \cref{t:DGK+converse}, and
\eqref{item2:toric-k^MJ} follows from \eqref{item1:toric-k^MJ}
and the observation that if $X$ is singular and $\Q$-factorial then 
$\s \cap N$ necessarily contains an exceptional indecomposable element. 
This is just because the exceptional locus of any resolution of singularity 
of a $\Q$-factorial variety has always pure codimension 1, 
and the set of essential (toric) valuations is nonempty if $X$ is singular.
\end{proof}

  

\end{document}

%% file: article-style.tex


\pagestyle{fancy}
\fancyhead{}
\fancyfoot{}
\fancyhead[LE,RO]{\small \thepage}
\fancyhead[RE]{\small \nouppercase{\rightmark}}
\fancyhead[LO]{\small \nouppercase{\leftmark}}

\setcounter{tocdepth}{1}

\makeatletter


\global\BR@BackrefAlttrue

\renewcommand*{\backrefalt}[4]{%
    \tiny%
    (%
    \ifcase #1 not cited%
          \or cit.~on~p.~#2%
          \else cit.~on~pp.~#2%
    \fi%
    )%
}

\def\print@backrefs#1{%
    \space\SentenceSpace%
    \begingroup%
        \expandafter\providecommand\csname brc@#1\endcsname{0}%
        \expandafter\providecommand\csname brcd@#1\endcsname{0}%
        \expandafter\backrefalt%
            \csname brc@#1\expandafter\endcsname%
            \csname brl@#1\expandafter\endcsname%
            \csname brcd@#1\expandafter\endcsname%
            \csname brld@#1\endcsname%
    \endgroup%
}

%

\def\maketitle{\par
  \@topnum\z@ 
  \@setcopyright
  \thispagestyle{empty}
  \ifx\@empty\shortauthors \let\shortauthors\shorttitle
  \else \andify\shortauthors
  \fi
  \@maketitle@hook
  \begingroup
  \@maketitle
  \toks@\@xp{\shortauthors}\@temptokena\@xp{\shorttitle}%
  \toks4{\def\\{ \ignorespaces}}
  \edef\@tempa{%
    \@nx\markboth{\the\toks4
      \@nx\MakeUppercase{\the\toks@}}{\the\@temptokena}}%
  \@tempa
  \endgroup
  \c@footnote\z@
    \renewcommand{\footnoterule}{%
      \kern -3pt
      \hrule width \textwidth height .5pt
      \kern 2pt
    }
  {
    \renewcommand\thefootnote{}
    \vspace{-2em}
    \footnote{
      \par\vspace{-1.2em}\noindent%
      \setlength{\parindent}{0pt}%
      \def\@footnotetext##1{\noindent{\footnotesize##1}\par}%
      \let\@makefnmark\relax  \let\@thefnmark\relax
      \ifx\@empty\@date\else \@footnotetext{\@setdate}\fi
      \ifx\@empty\@subjclass\else \@footnotetext{\@setsubjclass}\fi
      \ifx\@empty\@keywords\else \@footnotetext{\@setkeywords}\fi
      \ifx\@empty\thankses\else \@footnotetext{%
        \@setthanks}%
      \fi
    }
    \addtocounter{footnote}{-1}
  }
  \@cleartopmattertags
}

%

\def\@adminfootnotes{\@empty}

%

\def\@settitle{\begin{center}%
  \baselineskip14\p@\relax
    \bfseries
\Large
  \@title
  \end{center}%
}

%

\def\@setauthors{%
  \begingroup
  \def\thanks{\protect\thanks@warning}%
  \trivlist
  \centering\footnotesize \@topsep30\p@\relax
  \advance\@topsep by -\baselineskip
  \item\relax
  \author@andify\authors
  \def\\{\protect\linebreak}%
  \large{\authors}%
  \ifx\@empty\contribs
  \else
    ,\penalty-3 \space \@setcontribs
    \@closetoccontribs
  \fi
  \endtrivlist
  \endgroup
}

%

\def\@setaddresses{\par
  \nobreak \begingroup
\footnotesize
  \def\author##1{\end{minipage}\hskip 2em \begin{minipage}[t]{.5\textwidth minus
  1em}\raggedright%
    ~\\[2em]{\bf##1}\\[.5em]%
  }%
  \interlinepenalty\@M
  \def\address##1##2{\begingroup
    {\ignorespaces##2}\endgroup\\[.5em]}%
  \def\curraddr##1##2{\begingroup
    \@ifnotempty{##2}{\nobreak\indent\curraddrname
      \@ifnotempty{##1}{, \ignorespaces##1\unskip}\/:\space
      ##2\par}\endgroup}%
  \def\email##1##2{\begingroup
    \@ifnotempty{##2}{\nobreak\indent
      \@ifnotempty{##1}{, \ignorespaces##1\unskip}
      \ttfamily##2\par}\endgroup}%
  \def\urladdr##1##2{\begingroup
    \def~{\char`\~}%
    \@ifnotempty{##2}{\nobreak\indent\urladdrname
      \@ifnotempty{##1}{, \ignorespaces##1\unskip}\/:\space
      \ttfamily##2\par}\endgroup}%
  \setlength{\parindent}{0pt}%
  \vfill%
  {
  \hskip -2em%
  \begin{minipage}{0mm}
  \addresses
  \end{minipage}
  }
  \endgroup
}

%

\renewcommand{\author}[2][]{%
  \ifx\@empty\authors
    \gdef\authors{#2}%
    \g@addto@macro\addresses{\author{#2}}%
  \else
    \g@addto@macro\authors{\and#2}%
    \g@addto@macro\addresses{\author{#2}}%
  \fi
  \@ifnotempty{#1}{%
    \ifx\@empty\shortauthors
      \gdef\shortauthors{#1}%
    \else
      \g@addto@macro\shortauthors{\and#1}%
    \fi
  }%
}
\edef\author{\@nx\@dblarg
  \@xp\@nx\csname\string\author\endcsname}

%

\def\@secnumfont{\@empty}

%

\def\section{\@startsection{section}{1}%
  \z@{.7\linespacing\@plus\linespacing}{.5\linespacing}%
  {\large\bfseries\centering}}


\newcommand{\@doititle@doi}[1]{%
    \href%
        {https://doi.org/\csname bib'doi\endcsname}%
        {\textit{#1}}%
}

\newcommand{\@doititle@url}[1]{%
    \href%
        {\csname bib'url\endcsname}%
        {\textit{#1}}%
}

\newcommand{\@doititle@mr}[1]{{%
    \def\MR##1{##1}%
    \href%
        {http://www.ams.org/mathscinet-getitem?mr=\csname bib'review\endcsname}%
        {\textit{#1} [\def\MR##1{##1}\csname bib'review\endcsname]}%
}}

\newcommand{\@doititle}[1]{%
    \IfEmptyBibField{doi}{%
        \IfEmptyBibField{url}{%
            \IfEmptyBibField{review}{%
                \let\@tempa\textit
            }{%
                \let\@tempa\textit
            }%
        }{%
            \let\@tempa\@doititle@url
        }%
    }{%
        \let\@tempa\@doititle@doi
    }%
    \@tempa{#1}%
}

\BibSpec{article}{%
    +{}  {\PrintAuthors}                {author}
    +{,} { \@doititle}                  {title}
    +{.} { }                            {part}
    +{:} { \textit}                     {subtitle}
    +{,} { \PrintContributions}         {contribution}
    +{.} { \PrintPartials}              {partial}
    +{,} { }                            {journal}
    +{}  { \textbf}                     {volume}
    +{}  { \PrintDatePV}                {date}
    +{,} { \issuetext}                  {number}
    +{,} { \eprintpages}                {pages}
    +{,} { }                            {status}
    +{,} { available at \eprint}        {eprint}
    +{}  { \PrintTranslation}           {translation}
    +{;} { \PrintReprint}               {reprint}
    +{.} { }                            {note}
    +{.} {}                             {transition}
}

\BibSpec{book}{%
    +{}  {\PrintPrimary}                {transition}
    +{,} { \@doititle}                  {title}
    +{.} { }                            {part}
    +{:} { \textit}                     {subtitle}
    +{,} { \PrintEdition}               {edition}
    +{}  { \PrintEditorsB}              {editor}
    +{,} { \PrintTranslatorsC}          {translator}
    +{,} { \PrintContributions}         {contribution}
    +{,} { }                            {series}
    +{,} { \voltext}                    {volume}
    +{,} { }                            {publisher}
    +{,} { }                            {organization}
    +{,} { }                            {address}
    +{,} { \PrintDateB}                 {date}
    +{,} { }                            {status}
    +{}  { \parenthesize}               {language}
    +{}  { \PrintTranslation}           {translation}
    +{;} { \PrintReprint}               {reprint}
    +{.} { }                            {note}
    +{.} {}                             {transition}
}

\BibSpec{collection.article}{%
    +{}  {\PrintAuthors}                {author}
    +{,} { \@doititle}                  {title}
    +{.} { }                            {part}
    +{:} { \textit}                     {subtitle}
    +{,} { \PrintContributions}         {contribution}
    +{,} { \PrintConference}            {conference}
    +{}  {\PrintBook}                   {book}
    +{,} { }                            {booktitle}
    +{,} { \PrintDateB}                 {date}
    +{,} { pp.~}                        {pages}
    +{,} { }                            {status}
    +{,} { available at \eprint}        {eprint}
    +{}  { \parenthesize}               {language}
    +{;} { \PrintReprint}               {reprint}
    +{.} { }                            {note}
    +{.} {}                             {transition}
}

\makeatother

%% file: emb_codim_space_arcs_2021-1217.bbl
\begin{bibdiv}
\begin{biblist}


\let\^\circum

\hypersetup{urlcolor=black}

\newcommand\arxiv[1]{%
    \href%
        {https://arxiv.org/abs/#1}%
        {arXiv:#1}%
}

\renewcommand\url[1]{%
    \href%
        {#1}%
        {#1}%
}



\bib{EGAiv_1}{article}{
   label    = {EGA\,IV$_1$},
   author   = {Grothendieck, A.},
   title    = {\'{E}l\'{e}ments de g\'{e}om\'{e}trie alg\'{e}brique. IV. \'{E}tude locale des sch\'{e}mas et des morphismes de sch\'{e}mas. I},
   language = {French},
   journal  = {Inst. Hautes \'{E}tudes Sci. Publ. Math.},
   number   = {20},
   date     = {1964},
   pages    = {259},
   issn     = {0073-8301},
   review   = {\MR{0173675}},
}

 \bib{stacks-project}{article}{
    label  = {Stacks}
    title  = {The Stacks Project},
    note   = {\url{http://stacks.math.columbia.edu}},
    url    = {http://stacks.math.columbia.edu},
 }

\bib{AT09}{article}{
   author  = {Asgharzadeh, M.},
   author  = {Tousi, M.},
   title   = {On the notion of Cohen-Macaulayness for non-Noetherian rings},
   journal = {J. Algebra},
   volume  = {322},
   date    = {2009},
   number  = {7},
   pages   = {2297--2320},
   issn    = {0021-8693},
   review  = {\MR{2553201}},
   doi     = {10.1016/j.jalgebra.2009.06.017},
}

\bib{AM69}{book}{
   author    = {Atiyah, M. F.},
   author    = {Macdonald, I. G.},
   title     = {Introduction to commutative algebra},
   publisher = {Addison-Wesley Publishing Co., Reading, Mass.-London-Don Mills, Ont.},
   date      = {1969},
   pages     = {ix+128},
   review    = {\MR{0242802}},
}

\bib{Art69}{article}{
   author  = {Artin, M.},
   title   = {Algebraic approximation of structures over complete local rings},
   journal = {Inst. Hautes \'{E}tudes Sci. Publ. Math.},
   number  = {36},
   date    = {1969},
   pages   = {23--58},
   issn    = {0073-8301},
   review  = {\MR{268188}},
}

\bib{Bec90}{article}{
    AUTHOR = {Becker, T.},
     TITLE = {Stability and {B}uchberger criterion for standard bases in
              power series rings},
   JOURNAL = {J. Pure Appl. Algebra},
    VOLUME = {66},
      YEAR = {1990},
    NUMBER = {3},
     PAGES = {219--227},
      ISSN = {0022-4049},
       DOI = {10.1016/0022-4049(90)90028-G},
}

\bib{Bh16}{article}{
   author  = {Bhatt, B.},
   title   = {Algebraization and Tannaka duality},
   journal = {Camb. J. Math.},
   volume  = {4},
   date    = {2016},
   number  = {4},
   pages   = {403--461},
   issn    = {2168-0930},
   review  = {\MR{3572635}},
   doi     = {10.4310/CJM.2016.v4.n4.a1},
}

\bib{BK84}{article}{
   author  = {Bochnak, J.},
   author  = {Kucharz, W.},
   title   = {Local algebraicity of analytic sets},
   journal = {J. Reine Angew. Math.},
   volume  = {352},
   date    = {1984},
   pages   = {1--14},
   issn    = {0075-4102},
   review  = {\MR{758691}},
   doi     = {10.1515/crll.1984.352.1},
}

\bib{Bou72}{book}{
   author    = {Bourbaki, N.},
   title     = {Elements of mathematics. Commutative algebra},
   note      = {Translated from the French},
   publisher = {Hermann, Paris; Addison-Wesley Publishing Co., Reading, Mass.},
   date      = {1972},
   pages     = {xxiv+625},
   review    = {\MR{0360549}},
}

\bib{Bou74}{book}{
  author    = {Bourbaki, N.},
  title     = {Elements of mathematics. Algebra, Part I: Chapters 1-3},
  note      = {Translated from the French},
  publisher = {Hermann, Paris; Addison-Wesley Publishing Co., Reading Mass.},
  date      = {1974},
  pages     = {xxiii+709},
  review    = {\MR{0354207}},
}


\bib{BS17a}{article}{
   author  = {Bourqui, D.},
   author  = {Sebag, J.},
   title   = {The Drinfeld-Grinberg-Kazhdan theorem is false for singular arcs},
   journal = {J. Inst. Math. Jussieu},
   volume  = {16},
   date    = {2017},
   number  = {4},
   pages   = {879--885},
   issn    = {1474-7480},
   review  = {\MR{3680347}},
   doi     = {10.1017/S1474748015000341},
}


\bib{BS17b}{article}{
   author  = {Bourqui, D.},
   author  = {Sebag, J.},
   title   = {The Drinfeld-Grinberg-Kazhdan theorem for formal schemes and singularity theory},
   journal = {Confluentes Math.},
   volume  = {9},
   date    = {2017},
   number  = {1},
   pages   = {29--64},
   issn    = {1793-7434},
   review  = {\MR{3713815}},
   doi     = {10.5802/cml.35},
}

\bib{BS17c}{article}{
   author  = {Bourqui, D.},
   author  = {Sebag, J.},
   title   = {Smooth arcs on algebraic varieties},
   journal = {J. Singul.},
   volume  = {16},
   date    = {2017},
   pages   = {130--140},
   issn    = {1949-2006},
   review  = {\MR{3670512}},
}

\bib{BS19a}{article}{
   author  = {Bourqui, D.},
   author  = {Sebag, J.},
   title   = {Cancellation and regular derivations},
   journal = {J. Algebra Appl.},
   volume  = {18},
   number  = {09},
   date    = {2019},
}

\bib{BS19b}{article}{
   author  = {Bourqui, D.},
   author  = {Sebag, J.},
   title   = {Finite formal model of toric singularities},
   journal = {J. Math. Soc. Japan},
   volume  = {71},
   date    = {2019},
   pages   = {805--829},
 }

\bib{Bou}{article}{
   author = {Bouthier, A.},
   title  = {Cohomologie étale des espaces d'arcs},
   note   = {Preprint, \arxiv{1509.02203v6}},
   url    = {https://arxiv.org/abs/1509.02203v6},
}

\bib{BK17}{article}{
   author = {Bouthier, A.},
   author = {Kazhdan, D.},
   title  = {Faisceaux pervers sur les espaces d'arcs},
   note   = {Preprint, \arxiv{1509.02203v5}},
   url    = {https://arxiv.org/abs/1509.02203v5},
}

\bib{BNS16}{article}{
   author  = {Bouthier, A.},
   author  = {Ng\^{o}, B. C.},
   author  = {Sakellaridis, Y.},
   title   = {On the formal arc space of a reductive monoid},
   journal = {Amer. J. Math.},
   volume  = {138},
   date    = {2016},
   number  = {1},
   pages   = {81--108},
   issn    = {0002-9327},
   review  = {\MR{3462881}},
   doi     = {10.1353/ajm.2016.0004},
}

\bib{BH93}{book}{
   author    = {Bruns, W.},
   author    = {Herzog, J.},
   title     = {Cohen-Macaulay rings},
   series    = {Cambridge Studies in Advanced Mathematics},
   volume    = {39},
   publisher = {Cambridge University Press, Cambridge},
   date      = {1993},
   pages     = {xii+403},
   isbn      = {0-521-41068-1},
   review    = {\MR{1251956}},
}

\bib{Chi20}{thesis}{
   author = {Chiu, C.},
   title  = {Local geometry of the space of arcs},
   type   = {PhD dissertation},
   year   = {2020},
   school = {University of Vienna},
}

\bib{CH}{article}{
   author = {Chiu, C.},
   author = {Hauser, H.},
   title  = {On the formal neighborhood of degenerate arcs},
   note   = {Preprint, available at \url{https://homepage.univie.ac.at/herwig.hauser}},
   url    = {https://homepage.univie.ac.at/herwig.hauser},
}

\bib{dFD14}{article}{
   author={de Fernex, T.},
   author={Docampo, R.},
   title={Jacobian discrepancies and rational singularities},
   journal={J. Eur. Math. Soc. (JEMS)},
   volume={16},
   date={2014},
   number={1},
   pages={165--199},
   issn={1435-9855},
   review={\MR{3141731}},
   doi={10.4171/JEMS/430},
}


\bib{dFD}{article}{
    AUTHOR = {de Fernex, T.},
    AUTHOR = {Docampo, R.},
     TITLE = {Differentials on the arc space},
   JOURNAL = {Duke Math. J.},
    VOLUME = {169},
      YEAR = {2020},
    NUMBER = {2},
     PAGES = {353--396},
      ISSN = {0012-7094},
   MRCLASS = {14E18 (13N05 14F10)},
  MRNUMBER = {4057146},
       DOI = {10.1215/00127094-2019-0043},
       URL = {https://doi.org/10.1215/00127094-2019-0043},
}

\bib{dFEI08}{article}{
   author={de Fernex, T.},
   author={Ein, L.},
   author={Ishii, S.},
   title={Divisorial valuations via arcs},
   journal={Publ. Res. Inst. Math. Sci.},
   volume={44},
   date={2008},
   number={2},
   pages={425--448},
   issn={0034-5318},
   review={\MR{2426354 (2010d:14055)}},
   doi={10.2977/prims/1210167333},
}

\bib{DL99}{article}{
   author  = {Denef, J.},
   author  = {Loeser, F.},
   title   = {Germs of arcs on singular algebraic varieties and motivic integration},
   journal = {Invent. Math.},
   volume  = {135},
   date    = {1999},
   number  = {1},
   pages   = {201--232},
   issn    = {0020-9910},
   review  = {\MR{1664700}},
   doi     = {10.1007/s002220050284},
}

\bib{Dri}{article}{
   author  = {Drinfeld, V.},
   title   = {On the Grinberg--Kazhdan formal arc theorem},
   note    = {Preprint, \arxiv{math/0203263}},
   url     = {http://arxiv.org/abs/math/0203263},
}

\bib{Dri18}{article}{
   author  = {Drinfeld, V.},
   title   = {Grinberg--Kazhdan theorem and Newton groupoids},
   note    = {Preprint, \arxiv{1801.01046}},
   url     = {http://arxiv.org/abs/1801.01046},
}

\bib{EI15}{article}{
   author={Ein, L.},
   author={Ishii, S.},
   title={Singularities with respect to Mather-Jacobian discrepancies},
   conference={
      title={Commutative algebra and noncommutative algebraic geometry. Vol.
      II},
   },
   book={
      series={Math. Sci. Res. Inst. Publ.},
      volume={68},
      publisher={Cambridge Univ. Press, New York},
   },
   date={2015},
   pages={125--168},
   review={\MR{3496863}},
}

\bib{EM09}{article}{
   author  = {Ein, L.},
   author  = {Musta\c{t}\u{a}, M.},
   title   = {Generically finite morphisms and formal neighborhoods of arcs},
   journal = {Geom. Dedicata},
   volume  = {139},
   date    = {2009},
   pages   = {331--335},
   issn    = {0046-5755},
   review  = {\MR{2481855}},
   doi     = {10.1007/s10711-008-9320-7},
}

\bib{Eis95}{book}{
   author    = {Eisenbud, D.},
   title     = {Commutative algebra},
   series    = {Graduate Texts in Mathematics},
   volume    = {150},
   publisher = {Springer-Verlag, New York},
   date      = {1995},
   pages     = {xvi+785},
   isbn      = {0-387-94268-8},
   isbn      = {0-387-94269-6},
   review    = {\MR{1322960}},
   doi       = {10.1007/978-1-4612-5350-1},
}

\bib{GH93}{article}{
   author  = {Gilmer, R.},
   author  = {Heinzer, W.},
   title   = {Primary ideals with finitely generated radical in a commutative ring},
   journal = {Manuscripta Math.},
   volume  = {78},
   date    = {1993},
   number  = {2},
   pages   = {201--221},
   issn    = {0025-2611},
   review  = {\MR{1202161}},
   doi     = {10.1007/BF02599309},
}

\bib{Gla89}{book}{
   author    = {Glaz, S.},
   title     = {Commutative coherent rings},
   series    = {Lecture Notes in Mathematics},
   volume    = {1371},
   publisher = {Springer-Verlag, Berlin},
   date      = {1989},
   pages     = {xii+347},
   isbn      = {3-540-51115-6},
   review    = {\MR{999133}},
   doi       = {10.1007/BFb0084570},
}

\bib{GK00}{article}{
   author  = {Grinberg, M.},
   author  = {Kazhdan, D.},
   title   = {Versal deformations of formal arcs},
   journal = {Geom. Funct. Anal.},
   volume  = {10},
   date    = {2000},
   number  = {3},
   pages   = {543--555},
   issn    = {1016-443X},
   review  = {\MR{1779611}},
   doi     = {10.1007/PL00001628},
}

\bib{HW}{article}{
   author = {Hauser, H.},
   author = {Woblistin, S.},
     TITLE = {Arquile {V}arieties -- {V}arieties {C}onsisting of {P}ower
              {S}eries in a {S}ingle {V}ariable},
   JOURNAL = {Forum Math. Sigma},
    VOLUME = {9},
      YEAR = {2021},
     PAGES = {Paper No. e78},
       DOI = {10.1017/fms.2021.73},
       URL = {https://doi.org/10.1017/fms.2021.73},
}

\bib{Ish13}{article}{
   author={Ishii, S.},
   title={Mather discrepancy and the arc spaces},
   journal={Ann. Inst. Fourier (Grenoble)},
   volume={63},
   date={2013},
   number={1},
   pages={89--111},
   issn={0373-0956},
   review={\MR{3089196}},
   doi={10.5802/aif.2756},
}

\bib{Lec64}{article}{
   author={Lech, C.},
   title={Inequalities related to certain couples of local rings},
   journal={Acta Math.},
   volume={112},
   date={1964},
   pages={69--89},
   issn={0001-5962},
   review={\MR{161876}},
   doi={10.1007/BF02391765},
}

\bib{Mat89}{book}{
   author    = {Matsumura, H.},
   title     = {Commutative ring theory},
   series    = {Cambridge Studies in Advanced Mathematics},
   volume    = {8},
   edition   = {2},
   publisher = {Cambridge University Press, Cambridge},
   date      = {1989},
   pages     = {xiv+320},
   isbn      = {0-521-36764-6},
   review    = {\MR{1011461}},
}

 \bib{MR18}{article}{
    author  = {Mourtada, H.},
    author  = {Reguera, A. J.},
    title   = {Mather discrepancy as an embedding dimension in the space of arcs},
    journal = {Publ. Res. Inst. Math. Sci.},
    volume  = {54},
    date    = {2018},
    number  = {1},
    pages   = {105--139},
    issn    = {0034-5318},
    review  = {\MR{3749346}},
    doi     = {10.4171/PRIMS/54-1-4},
}

\bib{NS10}{article}{
   author  = {Nicaise, J.},
   author  = {Sebag, J.},
   title   = {Greenberg approximation and the geometry of arc spaces},
   journal = {Comm. Algebra},
   volume  = {38},
   date    = {2010},
   number  = {11},
   pages   = {4077--4096},
   issn    = {0092-7872},
   review  = {\MR{2764852}},
   doi     = {10.1080/00927870903295398},
}

\bib{Ngo17}{article}{
   label         = {Ngo},
   author        = {Ng\^{o}, B. C.},
   title         = {Weierstrass preparation theorem and singularities in the space of non-degenerate arcs},
   note          = {Preprint, \arxiv{1706.05926}},
   url           = {http://arxiv.org/abs/1706.05926},
}

\bib{Reg06}{article}{
   author  = {Reguera, A. J.},
   title   = {A curve selection lemma in spaces of arcs and the image of the Nash map},
   journal = {Compos. Math.},
   volume  = {142},
   date    = {2006},
   number  = {1},
   pages   = {119--130},
   issn    = {0010-437X},
   review  = {\MR{2197405}},
   doi     = {10.1112/S0010437X05001582},
}

\bib{Reg09}{article}{
   author  = {Reguera, A. J.},
   title   = {Towards the singular locus of the space of arcs},
   journal = {Amer. J. Math.},
   volume  = {131},
   date    = {2009},
   number  = {2},
   pages   = {313--350},
   issn    = {0002-9327},
   review  = {\MR{2503985}},
   doi     = {10.1353/ajm.0.0046},
}

\bib{Reg18}{article}{
   author  = {Reguera, A. J.},
   title   = {Coordinates at stable points of the space of arcs},
   journal = {J. Algebra},
   volume  = {494},
   date    = {2018},
   pages   = {40--76},
   issn    = {0021-8693},
   review  = {\MR{3723170}},
   doi     = {10.1016/j.jalgebra.2017.09.031},
}

\bib{Seb16}{article}{
   author  = {Sebag, J.},
   title   = {Primitive arcs on curves},
   journal = {Bull. Belg. Math. Soc. Simon Stevin},
   volume  = {23},
   date    = {2016},
   number  = {4},
   pages   = {481--486},
   issn    = {1370-1444},
   review  = {\MR{3579662}},
}


\bib{Whi65}{article}{
   author       = {Whitney, H.},
   title        = {Local properties of analytic varieties},
   conference   = {
      title     = {Differential and Combinatorial Topology (A Symposium in Honor of Marston Morse)},
   },
   book         = {
      publisher = {Princeton Univ. Press, Princeton, N. J.},
   },
   date         = {1965},
   pages        = {205--244},
   review       = {\MR{0188486}},
}

 \bib{Voj07}{article}{
    author       = {Vojta, P.},
    title        = {Jets via Hasse-Schmidt derivations},
    conference   = {
       title     = {Diophantine geometry},
    },
    book         = {
       series    = {CRM Series},
       volume    = {4},
       publisher = {Ed. Norm., Pisa},
    },
    date         = {2007},
    pages        = {335--361},
    review       = {\MR{2349665}},
}


\end{biblist}
\end{bibdiv}
